\documentclass[12pt]{amsart}
 
\pdfoutput=1
\usepackage[margin=1in]{geometry}
\usepackage{amsmath,amsthm,amssymb,amsrefs,bbm,color,enumitem,esint,esvect,float,graphicx,mathrsfs}
\usepackage{euscript,latexsym}
\usepackage{accents}

\usepackage{epstopdf}
\AppendGraphicsExtensions{.tif}

\newtheorem{thm}{Theorem}[section]
\newtheorem{lemma}[thm]{Lemma}
\newtheorem{cor}[thm]{Corollary}
\newtheorem{prop}[thm]{Proposition}
\newtheorem{rem}[thm]{Remark}

\newtheorem{open}[thm]{Open Question}
 
\numberwithin{equation}{section}
 
\newenvironment{theorem}[2][Theorem]{\begin{trivlist}
\item[\hskip \labelsep {\bfseries #1}\hskip \labelsep {\bfseries #2.}]}{\end{trivlist}}

\theoremstyle{definition}

\begin{document}

\title[Weighted theory of Toeplitz operators on the Bergman space]{Weighted theory of Toeplitz operators on the Bergman space}

\author{Cody B. Stockdale}
\address{Cody B. Stockdale\hfill\break\indent 
 School of Mathematical Sciences and Statistics\hfill\break\indent 
 Clemson University\hfill\break\indent 
 105 Sikes Hall\hfill\break\indent 
 Clemson, SC 29634 USA}
\email{cbstock@clemson.edu}

\author{Nathan A. Wagner}
\address{Nathan A. Wagner \hfill\break\indent 
 Department of Mathematics \hfill\break\indent 
Brown University \hfill\break\indent 
151 Thayer Street \hfill\break\indent 
 Providence, RI 02912 USA}
\email{nathan\_wagner@brown.edu}
\thanks{The second author is the corresponding author}
\thanks{N. A. Wagner's research was supported in part by National Science Foundation grant DGE \#1745038.}

\begin{abstract}
We study the weighted compactness and boundedness properties of Toeplitz operators on the Bergman space with respect to B\'ekoll\`e-Bonami type weights. Let $T_u$ denote the Toeplitz operator on the (unweighted) Bergman space of the unit ball in $\mathbb{C}^n$ with symbol $u \in L^{\infty}$. 
We characterize the compact Toeplitz operators on the weighted Bergman space $\mathcal{A}^p_\sigma$ for all $\sigma$ in a subclass of the B\'ekoll\`e-Bonami class $B_p$ that includes radial weights and powers of the Jacobian of biholomorphic mappings. Concerning boundedness, we show that $T_u$ extends boundedly on $L^p_{\sigma}$ for $p \in (1,\infty)$ and weights $\sigma$ in a $u$-adapted class of weights containing $B_p$, and we establish analogous weighted endpoint weak-type $(1,1)$ bounds for weights beyond $B_1$. 



\end{abstract}

\maketitle


\section{Introduction}\label{Introduction}

The Bergman space of the unit ball $\mathbb{B}_n\subseteq \mathbb{C}^n$ is defined to be
$$
    \mathcal{A}^2:= L^2 \cap \text{Hol}(\mathbb{B}_n)
$$
where $L^2$ denotes the space of square integrable functions on $\mathbb{B}_n$ with respect to normalized Lebesgue measure $V$ and $\text{Hol}(\mathbb{B}_n)$ represents the holomorphic functions on $\mathbb{B}_n$. Since $\mathcal{A}^2$ is a closed subspace of $L^2$, there exists an orthogonal projection from $L^2$ onto $\mathcal{A}^2$. This map is called the Bergman projection and is given by 
$$
    Pf(z)=\int_{\mathbb{B}_n} K(z,w)f(w)\,dV(w),
$$
where $K$ is the reproducing kernel of $\mathcal{A}^2$. Recall that 
$$
    K(z,w)=K_w(z)=\frac{1}{(1-z\overline{w})^{n+1}},
$$
where $z\overline{w}:=\sum_{j=1}^{n}z_j\overline{w}_j$.

It is clear that the Bergman projection is a bounded operator on $L^2$; in fact, $P$ acts boundedly on $L^p$ for all $p \in (1,\infty)$. Concerning the endpoint behavior, $P$ fails to be bounded on $L^1$, but instead satisfies the weak-type $(1,1)$ inequality: there exists $C>0$ such that
$$
    \|Pf\|_{L^{1,\infty}}:=\sup_{\lambda>0}\lambda|\{z \in \mathbb{B}_n: |Pf(z)|>\lambda\}|\leq C\|f\|_{L^1}
$$
for all $f \in L^1$, where $|A|$ represents the Lebesgue measure of $A\subseteq \mathbb{B}_n$. We write $L^{1,\infty}$ for the space of functions $f$ such that $\|f\|_{L^{1,\infty}}$ as defined above is finite and put $\mathcal{A}^{1,\infty}:=L^{1,\infty}\cap\text{Hol}(\mathbb{B}_n)$. We denote the smallest constant $C$ in the weak-type inequality above by $\|P\|_{L^1\rightarrow L^{1,\infty}}$ and the operator norms of $P$ on $L^p$ by $\|P\|_{L^p\rightarrow L^p}$. See \cites{Z2020} for a recent survey on the $L^p$ theory of $P$ and \cites{DHZZ2001, B198182,M1994,SW2020} for work on the weak-type $(1,1)$ bound for $P$.

Given a function $u$ on $\mathbb{B}_n$, the associated Toeplitz operator, $T_u$, is given by
\begin{align*}
    T_{u}f(z):=P(uf)(z)=\int_{\mathbb{B}_n}K(z,w)u(w)f(w)\,dV(w).
\end{align*}
Appealing to the properties of the Bergman projection, it immediate to see that if $u \in L^{\infty}$, then $T_u$ extends boundedly on $L^p$ for $p \in (1,\infty)$ and from $L^1$ to $L^{1,\infty}$. Also, if $f \in L^1$, then $T_uf$ is holomorphic (see \cite{Ru1980}), and so $T_u$ actually maps $L^p$ to $\mathcal{A}^p$ and $L^1$ to $\mathcal{A}^{1,\infty}$ boundedly. 

Moreover, if the symbol $u$ is ``nicer" than a general $L^{\infty}$ function, then these properties can be strengthened. For example, if the symbol decays in an appropriate sense, then the associated Toeplitz operator may be compact, rather than just bounded. Recall that a linear operator $T$ is compact from a Banach space $\mathcal{X}$ to a topological vector space $\mathcal{Y}$ if $T(\overline{A})$ is precompact in $\mathcal{Y}$ whenever $A$ is a bounded subset of $\mathcal{X}$, and that compact operators are necessarily bounded.

In \cite{AZ1998}, Axler and Zheng characterized the Toeplitz operators with bounded symbols that act compactly on $\mathcal{A}^2$. Their characterization involves the Berezin transform 
$$
    \widetilde{T_u}(z):=\langle T_uk_z,k_z\rangle,
$$
where the $k_z$ are the normalized reproducing kernels for $\mathcal{A}^2$ given by 
$$
    k_z(w)=\frac{(1-|z|^2)^{\frac{n+1}{2}}}{(1-\overline{z}w)^{n+1}}.
$$ 
The work of Axler and Zheng originally addressed the one dimensional situation and has since seen many generalizations, see \cites{E1999, MSW2013,IMW2015,S2007,WX2021}. In \cite{E1999}, Engli\v{s} obtained a generalization of Axler and Zheng's result for Bergman spaces of bounded symmetric domains in $\mathbb{C}^n$, which includes the unit ball. A version of the result is stated as follows. 
\begin{theorem}{A} \label{AZ}\emph{
Let $u \in L^{\infty}$ and $p \in (1,\infty)$. Then $T_u$ acts compactly on $\mathcal{A}^p$ if and only if $\widetilde{T_u}(z) \rightarrow 0$ as $|z|\rightarrow 1^-$.
}
\end{theorem}

Recent attention has been given to weighted bounds for the Bergman projection. We say that a function $\sigma$ on $\mathbb{B}_n$ is a weight if it is locally integrable and positive almost everywhere. Given a weight $\sigma$ and $p \in [1,\infty)$, we denote by $L^p_{\sigma}$ the space of functions $f$ for which 
$$
    \|f\|_{L^p_{\sigma}}^p:=\int_{\mathbb{B}_n}|f|^p\sigma\,dV<\infty
$$
and write $L^{1,\infty}_{\sigma}$
for the space of functions $f$ for which 
$$
    \|f\|_{L^{1,\infty}_{\sigma}}:=\sup_{\lambda >0}\lambda\sigma(\{z\in\mathbb{B}_n : |f(z)|>\lambda\})<\infty,
$$
where $\sigma(A)$ denotes $\int_A \sigma \,dV$for $A\subseteq\mathbb{B}_n$. We write $\mathcal{A}^p_{\sigma}$ for the space of holomorphic $f \in L^p_{\sigma}$ and and $\mathcal{A}^{1,\infty}_{\sigma}$ for the space of holomorphic $f \in L^{1,\infty}_{\sigma}$. 

Extending the work of B\'ekoll\`e and Bonami from \cite{BB1978}, in \cite{B198182}, B\'ekoll\`e characterized the weights $\sigma$ for which the Bergman projection acts boundedly on $L^p_{\sigma}$ for $p \in (1,\infty)$ and from $L^1_{\sigma}$ to $L^{1,\infty}_{\sigma}$. Precisely, the following was proved. 
\begin{theorem}{B} 
\emph{The Bergman projection acts boundedly on $L^p_{\sigma}$ for $p\in (1,\infty)$ if and only if 
$$
    [\sigma]_{B_p}:=\sup_{z \in \mathbb{B}_n} \langle \sigma\rangle_{\mathcal{T}_z}\langle \sigma^{1-p'}\rangle_{\mathcal{T}_z}^{p-1}<\infty
$$
and $P$ acts boundedly from $L^1_{\sigma}$ to $L^{1,\infty}_{\sigma}$ if and only if 
$$
    [\sigma]_{B_1}:=\sup_{z\in\mathbb{B}_n}\langle \sigma\rangle_{\mathcal{T}_z}\|\sigma^{-1}\|_{L^{\infty}(\mathcal{T}_z)}<\infty.
$$}
\end{theorem}
Above, $\langle \sigma\rangle_{A}$ represents the average $\frac{\sigma(A)}{|A|}$ for $A\subseteq \mathbb{B}_n$, and $\mathcal{T}_z$ denotes the Carleson tent over $z \in \mathbb{B}_n$ given by
$$
    \mathcal{T}_z:=\left\{w \in \mathbb{B}_n: \left| 1 - \overline{w}\frac{z}{|z|}\right|<1-|z|\right\}.
$$
If $z=0,$ we interpret $\mathcal {T}_z$ to be $\mathbb{B}_n.$
If $[\sigma]_{B_p}<\infty$, we write $\sigma \in B_p$ and refer to $[\sigma]_{B_p}$ as the B\'ekoll\`e-Bonami characteristic of $\sigma$. See \cites{RTW2017,WW2020,HWW20201,HWW20202} for recent weighted $L^p_{\sigma}$ theory of $P$ and to \cite{SW2020} for recent weighted endpoint weak-type theory of $P$.

Appealing to the weighted estimates of the Bergman projection, it is again immediate to see that if $u \in L^{\infty}$, then $T_u$ extends boundedly on $L^p_{\sigma}$ for $p \in (1,\infty)$ and all $\sigma \in B_p$, and from $L^1_{\sigma}$ to $L^{1,\infty}_{\sigma}$ for all $\sigma \in B_1$. 

The purpose of this paper is to investigate the weighted properties of Toeplitz operators on the Bergman space of $\mathbb{B}_n$. We consider both the perspective of these operators acting on $L^p_\sigma$ as well as $\mathcal{A}^p_\sigma.$ The latter 
perspective has been typical in the study of compactness of Toeplitz operators, however Toeplitz operators with decaying symbols have also been 
considered as operators that ``improve" $L^p$ spaces, see \cite{CM2006}. By viewing Toeplitz operators as integral operators from a harmonic analysis perspective, we study their boundedness and compactness on $L^p_\sigma$ as well as $\mathcal{A}^p_{\sigma}$. 

We investigate versions of Theorem A with B\'ekoll\`e-Bonami weights and Theorem B for Toeplitz operators. We develop improvements for Toeplitz operators in the following ways: 
\begin{enumerate}
\addtolength{\itemsep}{0.2cm}
    \item we characterize the compact Toeplitz operators on $\mathcal{A}_{\sigma}^p$ for all $p \in (1,\infty)$ and all $\sigma$ in the intersection of $B_p$ and a reverse H\"older class,
    \item we show that if $u$ decays at the boundary of $\mathbb{B}_n$, then $T_u$ extends boundedly on $L^p_{\sigma}$ for $p \in (1,\infty)$ and $\sigma$ in a class of weights properly containing $B_p$, and
    \item we prove that if $u$ decays at the boundary of $\mathbb{B}_n$, then $T_u$ extends boundedly from $L^1_{\sigma}$ to $L^{1,\infty}_{\sigma}$ for $\sigma$ in a class of weights properly containing the intersection of $B_1$ and a reverse H\"older class.
\end{enumerate}  

In previous research considering Toeplitz operators on weighted Bergman spaces, the Toeplitz operators are defined using the weighted Bergman projection, and the Berezin transform is defined using the normalized kernels from the weighted space. Boundedness and compactness results of this kind have been obtained for Toeplitz operators with nonnegative symbols, see \cites{Ch2013,Co2010,CSZ2018}. In contrast, our results concern Toeplitz operators defined using reproducing kernels from the unweighted Bergman space, akin to how the unweighted Bergman projection has been studied on $L^p_{\sigma}$ by B\'ekoll\`e and Bonami. 

We state and discuss our main results in the following Subsection \ref{MainResults}, and we discuss their straightforward extensions and some open problems in Subsection \ref{Extensions}.

The authors thank Walton Green, Mishko Mitkovski, Brett Wick, and Yunus Zeytuncu for their valuable discussions.

\subsection{Main Results}\label{MainResults}
Our first stated result is likely known to experts, however, we include it here for motivation and completeness.
\begin{thm}\label{ToeplitzLpCompactness}
Let $u \in L^{\infty}$ and $p \in (1,\infty)$. If $u(z)\rightarrow 0$ as $|z|\rightarrow 1^-$, then $T_u$ acts compactly on $L^p_{\sigma}$ for all $\sigma \in B_p$ and $T_u$ acts compactly from $L^1_{\sigma}$ to $L^{1,\infty}_{\sigma}$ for all $\sigma \in B_1$.
\end{thm}
\noindent Our proof of Theorem \ref{ToeplitzLpCompactness} relies on compactness characterizations for subsets of $\mathcal{A}_{\sigma}^p$ and $\mathcal{A}^{1,\infty}_{\sigma}$. 
In certain situations, the condition of Theorem \ref{ToeplitzLpCompactness} is implied by the vanishing Berezin transform condition of Theorem A. For instance, we have the following corollary.
\begin{cor}\label{ContinuousLpCompactness}
Let $u \in L^{\infty}$ be continuous on $\overline{\mathbb{B}}_n\setminus K$ for some compact $K \subseteq \mathbb{B}_n$. The following are equivalent:
\begin{enumerate}[label=(\alph*)]
\addtolength{\itemsep}{0.2cm}
    \item $T_u$ acts compactly on $L^p$ for some $p \in (1,\infty)$, 
    \item $T_u$ acts compactly on $L^p_{\sigma}$ for all $p \in (1,\infty)$ and all $\sigma \in B_p$,
    \item $T_u$ acts compactly from $L^1_{\sigma}$ to $L^{1,\infty}_{\sigma}$ for all $\sigma \in B_1$, and 
    \item $\widetilde{T_u}(z)\rightarrow 0$ as $|z| \rightarrow 1^-$. 
\end{enumerate}
\end{cor}



To phrase our next results, we first discuss a dyadic structure on $\mathbb{B}_n$ consisting of ``kubes," $K$, and their associated tents, $\widehat{K}$. In fact, we work with a fixed finite collection of dyadic structures, $\mathcal{D}_1, \ldots, \mathcal{D}_M.$ We follow the construction in \cites{ARS2006,RTW2017} and provide these details in Section \ref{Preliminaries}. Put $\mathcal{D}:=\bigcup_{\ell=1}^{M} \mathcal{D}_{\ell}$. For $p\in (1,\infty)$ and $\sigma \in B_p$, set
$$
    [\sigma]_{B_p^{d}}:=\sup_{K \in \mathcal{D}}\langle \sigma\rangle_{\widehat{K}}\langle \sigma^{1-p'}\rangle_{\widehat{K}}^{p-1}
$$
and for $\sigma \in B_1$, set
$$
    [\sigma]_{B_1^{d}}:=\sup_{K\in\mathcal{D}}\langle \sigma\rangle_{\widehat{K}}\|\sigma^{-1}\|_{L^{\infty}(\widehat{K})}.
$$
We will see in Section \ref{Preliminaries} that a weight $\sigma$ is in $B_p$ for $p \in [1,\infty)$ if and only if $[\sigma]_{B_p^d}<\infty$, and that $[\sigma]_{B_p}$ is comparable to $[\sigma]_{B_p^d}$. 

Assuming that our weights $\sigma$ satisfy a reverse H\"older condition, we obtain a full weighted Axler-Zheng type characterization of the compact Toeplitz operators on $\mathcal{A}_{\sigma}^p$ for $p \in (1,\infty)$ via the Berezin transform. In particular, we introduce the following reverse H\"older class: for $r>1$, we say that a weight $\sigma$ is in the reverse H\"older class $\text{RH}_r$ and write $\sigma \in \text{RH}_{r}$ if 
$$
    [\sigma]_{\text{RH}_r}:=\sup_{K \in \mathcal{D}} \frac{\langle\sigma^r\rangle_{\widehat{K}}^{1/r}}{\langle\sigma\rangle_{\widehat{K}}}<\infty.
$$

\begin{thm}\label{LpCompactnessCharacterization}
Let $u \in L^\infty$, $p \in (1,\infty)$, $r>1$, and $\sigma \in B_p \cap \text{RH}_{r}$ with $\sigma^{1-p'} \in \text{RH}_r $. Then $T_u$ acts compactly on $\mathcal{A}_{\sigma}^p$ if and only if $\widetilde{T}_u(z)\rightarrow 0$ as $|z|\rightarrow 1^-$.
\end{thm}

Theorem \ref{LpCompactnessCharacterization} generalizes Theorem A in a major way to include weighted Bergman spaces for some important examples of weights, such as those considered by Aleman, Pott, and Reguera in \cite{APR2019}, which we now discuss. Let $M$ denote the maximal operator defined by
$$
    Mf(z):= \sup_{K \in \mathcal{D}}\langle |f|\rangle_{\widehat{K}}\chi_{\widehat{K}}(z).
$$
We say $\sigma \in B_{\infty}$ if 
$$
    [\sigma]_{B_{\infty}}:= \sup_{K \in \mathcal{D}} \frac{1}{\sigma(\widehat{K})}\int_{\widehat{K}}M(\sigma \chi_{\widehat{K}})\,dV < \infty.
$$
Note that $B_{p_1} \subseteq B_{p_2}$ for all $1\leq p_1\leq p_2\leq \infty$. Let $\phi_z$ be the involutive automorphism of $\mathbb{B}_n$ interchanging $z$ and $0$ and define the Bergman metric, $\beta$, on $\mathbb{B}_n\times\mathbb{B}_n$ by 
$$
    \beta(z,w):=\frac{1}{2}\log\left(\frac{1+|\phi_z(w)|}{1-|\phi_z(w)|}\right).
$$
We denote the ball centered at $z$ of radius $r$ with respect to $\beta$ by $B_{\beta}(z,r)$. We consider weights $\sigma$ satisfying the following condition: for all $r>0$, there exists $c_{\sigma,r}>0$ such that
\begin{align}\label{WeightRegularity}
  \sigma(z) \leq c_{\sigma,r}\sigma(w) \quad\quad \text{for all } \zeta \in \mathbb{B}_n \,\, \text{and all} \,\, z,w \in B_{\beta}(\zeta,r).
\end{align}
Condition \eqref{WeightRegularity} was first introduced on the unit disk in \cite{APR2019} where it was proved that if $\sigma \in B_{\infty}$ satisfies \eqref{WeightRegularity}, then $\sigma \in \text{RH}_r$ for some $r>1$. In fact, we will see in Section \ref{Preliminaries} that \eqref{WeightRegularity} implies the existence of $c_{\sigma}>0$ such that
\begin{align}\label{DyadicWeightRegularity}
    \sigma(z)\leq c_{\sigma}\sigma(w)\quad\quad \text{for all } K\in \mathcal{D} \,\, \text{and all} \,\, z,w \in K.
\end{align}

We therefore have the following immediate corollary of Theorem \ref{LpCompactnessCharacterization}. 
\begin{cor}\label{LpCompactnessCharacterizationCorollary}
Let $u \in L^{\infty}$, $p \in (1,\infty)$, and $\sigma \in B_p$ satisfy \eqref{DyadicWeightRegularity}. Then $T_u$ extends compactly on $\mathcal{A}_{\sigma}^p$ if and only if $\widetilde{T}_u(z)\rightarrow 0$ as $|z|\rightarrow 1^-$.
\end{cor}

Corollary \ref{LpCompactnessCharacterizationCorollary} applies to weighted Bergman spaces with the standard radial weights. For $b>-1$ and normalizing constant $c_b$, we define the radial weight 
$$
    \sigma_b(z):=c_b(1-|z|^2)^{b}.
$$
The weighted Bergman spaces $\mathcal{A}^p_{\sigma_b}$ appear naturally in function theory and operator theory, see \cite{Zhu1}. It is well-known that $\sigma_b \in B_p$ as long as $b \in (-1,p-1)$, and it is straightforward to show that such $\sigma_b$ satisfy the hypotheses of Theorem \ref{LpCompactnessCharacterization} as well as \eqref{WeightRegularity}. 

\begin{cor}\label{RadialAz}
Let $u \in L^\infty.$ The following are equivalent:
\begin{enumerate}[label=(\alph*)]
\addtolength{\itemsep}{0.2cm}
\item $T_u$ acts compactly on $\mathcal{A}^p_{\sigma_b}$ for some $p \in (1,\infty)$ and some $b \in (-1,p-1)$, 
\item $T_u$ acts compactly on $\mathcal{A}^p_{\sigma_b}$ for all $p \in (1,\infty)$ and all $b \in (-1,p-1)$, and
\item $\widetilde{T}_u(z)\rightarrow 0$ as $|z|\rightarrow 1^{-}$.
\end{enumerate}
\end{cor}

Corollary \ref{LpCompactnessCharacterizationCorollary} also includes weights given by powers of the Jacobian of biholomorphic mappings on $\mathbb{D}:=\mathbb{B}_1$. Such weights play a crucial role in proving $L^p$ regularity for the (unweighted) Bergman projection on simply connected planar domains, see \cite{LS2004} for more information. These weights can be seen to satisfy \eqref{WeightRegularity} 
using the Koebe distortion theorem and the main theorem in \cite{APR2019}. We have the following additional corollary concerning such weights. Below, $g$ denotes a univalent function on $\mathbb{D}$ and $\sigma_g^a:=|g'|^a$.

\begin{cor}\label{ConformalAz}
Let $u \in L^\infty$. The following are equivalent:
\begin{enumerate}[label=(\alph*)]
\addtolength{\itemsep}{0.2cm}
\item $T_u$ acts compactly on $\mathcal{A}_{\sigma_g^a}^p$ for some $p \in (1,\infty)$ and some $\sigma_g^a \in B_p$ such that $g$ is univalent on $\mathbb{D}$ and $a>0$, 
\item $T_u$ acts compactly on $\mathcal{A}_{\sigma_g^a}$ for all $p \in (1,\infty)$ and all $\sigma_g^a \in B_p$ such that $g$ is univalent on $\mathbb{D}$ and $a>0$, and 
\item $\widetilde{T}_u(z)\rightarrow 0$ as $|z|\rightarrow 1^{-}$.
\end{enumerate}
\end{cor}

Theorem \ref{LpCompactnessCharacterization} follows from an extrapolation of compactness argument, adapted from the recent work of Hyt\"onen and Lappas from \cite{HL2020}. Specifically, we prove the following.
\begin{thm}\label{CompactnessExtrapolation}
Let $T$ be a linear operator, $p_0 \in (1,\infty)$, and $r>1$. If $T$ is bounded on $\mathcal{A}_{\sigma}^{p_0}$ for all $\sigma \in B_{p_0}$ and $T$ is compact on $\mathcal{A}_{\sigma_0}^{p_0}$ for some $\sigma_0 \in B_{p_0}\cap \text{RH}_r$ with $\sigma_0^{1-p_0'} \in \text{RH}_r$, then $T$ is compact on $\mathcal{A}_{\sigma}^p$ for all $p \in (1,\infty)$ and all $\sigma \in B_p \cap \text{RH}_r$ with $\sigma^{1-p'} \in \text{RH}_r$.
\end{thm}

Given $u \in L^{\infty}$, we say that a weight $\sigma$ is in $u B_p$ for $p \in (1,\infty)$ if
$$
    [\sigma]_{uB_p}:=\sup_{K \in \mathcal{D}}\|u\|_{L^{\infty}(\widehat{K})}\langle \sigma\rangle_{\widehat{K}}\langle \sigma^{1-p'}\rangle_{\widehat{K}}^{p-1}<\infty
$$
and $\sigma$ is in $u B_1$ if
$$
    [\sigma]_{uB_1}:=\sup_{K\in\mathcal{D}}\|u\|_{L^{\infty}(\widehat{K})}\langle \sigma\rangle_{\widehat{K}}\|\sigma^{-1}\|_{L^{\infty}(\widehat{K})}<\infty.
$$
Notice that $[\sigma]_{u B_p}\leq \|u\|_{L^{\infty}}[\sigma]_{B_p^d}$, and therefore $B_p \subseteq u B_p$. Further, if $u^{-1}$ is also in $L^{\infty}$, then $[\sigma]_{B_p^d} \leq \|u^{-1}\|_{L^{\infty}}[\sigma]_{u B_p}$ and so $B_p = uB_p$. However, if $\|u\|_{L^{\infty}(\widehat{K}_j)}\rightarrow 0$ as $j\rightarrow \infty$ for some sequence  $\{K_j\}_{j=1}^{\infty}\subseteq \mathcal{D}$, then 
$$
    B_p \subsetneq u B_p.
$$
Indeed, in this case, the products $\langle \sigma\rangle_{\widehat{K}_j}\langle \sigma^{1-p'}\rangle_{\widehat{K}_j}^{p-1}$ could be as large as $\|u\|_{L^{\infty}(\widehat{K}_j)}^{-1}$.


We obtain weighted boundedness results for weights beyond the $B_p$ classes.
\begin{thm}\label{ToeplitzLpBoundedness}

Let $u \in L^{\infty}$ and $p \in (1,\infty)$. If $\sigma \in u^{\min(1,p-1)}B_p$, then $T_{u}$ acts boundedly on $L^p_{\sigma}$ with 
$$
    \|T_u\|_{L^p_{\sigma}\rightarrow L^p_{\sigma}}\leq C [\sigma]_{u^{\min(p-1,1)}B_p}^{\max(1,\frac{1}{p-1})}
$$
for some $C>0$.
\end{thm}
\noindent Theorem \ref{ToeplitzLpBoundedness} uses a sparse domination technique and parallels the result for Haar multipliers of \cite{SVW2019}*{Theorem 1.3}.

Assuming the reverse H\"older condition, we obtain a weighted weak-type $(1,1)$ result for Toeplitz operators with weights beyond $B_1$.
\begin{thm}\label{ToeplitzL1Boundedness}
Let $u \in L^{\infty}$ and $r>1$. If $\sigma \in u B_1 \cap \text{RH}_r$, then $T_u$ acts boundedly from $L^1_{\sigma}$ to $L^{1,\infty}_{\sigma}$ with
$$
    \|T_u\|_{L^1_{\sigma}\rightarrow L^{1,\infty}_{\sigma}}\leq C [\sigma]_{uB_1}[\sigma]_{\text{RH}_r}(1+\log r')
$$
for some $C>0$.
\end{thm}

The following quantitative weighted weak-type estimate follows from Theorem \ref{ToeplitzL1Boundedness} and our improved reverse H\"older inequality. 


\begin{cor}\label{ToeplitzL1BoundednessCorollary}
If $u \in L^{\infty}$ and $\sigma \in u B_1 \cap B_{\infty}$ satisfies \eqref{DyadicWeightRegularity}, then $T_u$ acts boundedly from $L^1_{\sigma}$ to $L^{1,\infty}_{\sigma}$ with
$$
    \|T_u\|_{L^1_{\sigma}\rightarrow L^{1,\infty}_{\sigma}}\leq C[\sigma]_{uB_1} c_{\sigma}^{\frac{1}{2\alpha [\sigma]_{B_\infty}+1}}\log(e+[\sigma]_{B_{\infty}})
$$
for some $C>0$.
\end{cor}
\noindent The constant $\alpha$ above is a ``dyadic doubling" constant that we describe in Section \ref{Preliminaries}.

\begin{rem}
Corollary \ref{ToeplitzL1BoundednessCorollary} in the case $u=1$ gives the following quantitative estimate for the Bergman projection: if $\sigma \in B_1$ satisfies \eqref{DyadicWeightRegularity}, then 
$$
    \|P\|_{L^1_{\sigma}\rightarrow L^{1,\infty}_{\sigma}}\leq C[\sigma]_{B_1}c_{\sigma}^{\frac{1}{2\alpha [\sigma]_{B_\infty}+1}}\log(e+[\sigma]_{B_{\infty}}).
$$
This improves the dependence on the B\'ekoll\`e-Bonami characteristics in the weighted weak-type $(1,1)$ bound for $P$ when considering weights that satisfy $\eqref{DyadicWeightRegularity}$ with bounded $c_{\sigma}$. Indeed, the only proofs of the weighted weak-type $(1,1)$ inequality for $P$ with $\sigma \in B_1$ exist in \cites{B198182,SW2020} -- the proof handling the generality of \cite{SW2020} only allows for a qualitative bound, while tracking the dependence in the proof of \cite{B198182} gives that 
$$
    \|P\|_{L^1_{\sigma}\rightarrow L^{1,\infty}_{\sigma}}\leq C[\sigma]_{B_1}^3
$$
for some $C>0$. We note that the proof in \cite{B198182} with some minor modifications can improve the dependence to 
$$
    \|P\|_{L^1_{\sigma}\rightarrow L^{1,\infty}_{\sigma}}\leq C[\sigma]_{B_1}^2
$$
for some $C>0$. 
\end{rem}

The above comments lead to the following corollary for radial weights, as a straightforward computation shows that $[\sigma_b]_{B_1} \approx \frac{1}{1+b}$ if $-1<b \leq 0$. In particular, this estimate provides quantitative control on the blowup of $\|P\|_{L_{\sigma_b}^1\rightarrow L_{\sigma_b}^{1,\infty}}$ as $b \rightarrow -1^{+}$, which to the authors' best knowledge was not previously known. 
\begin{cor}
Let $b \in (-1,0].$ The following weighted weak-type estimate for the (unweighted) Bergman projection with radial weights holds:
$$       
\|P\|_{L^1_{\sigma_b}\rightarrow L^{1,\infty}_{\sigma_b}}\leq \frac{C}{1+b}\log\left(e+\frac{1}{1+b}\right),
$$
for some $C>0$.

\end{cor}


\subsection{Extensions and Open Problems}\label{Extensions}
We fix the underlying domain of our spaces to be $\mathbb{B}_n$ for simplicity; however, our results extend to handle spaces of functions defined on more general domains. We may define weighted Bergman and Lebesgue spaces of a general domain $D\subseteq \mathbb{C}_n$ ($\mathcal{A}^p_{\sigma}(D)$, $L^p_{\sigma}(D)$, $\mathcal{A}^{1,\infty}_{\sigma}(D)$, and $L^{1,\infty}_{\sigma}(D)$) in the same way as above, except with $\mathbb{B}_n$ replaced by $D$. Other components of our statements, such as $B_p$ weights, conditions on symbols, dyadic structures, etcetera, may be understood in these contexts with suitable changes. We discuss how our arguments and results can be extended to spaces of more general domains.

We call a domain $D$ simple if it is either
\begin{enumerate}
    \item a finite type domain in $\mathbb{C}^2$, 
    \item a decoupled, finite type domain in $\mathbb{C}^n$,
    \item a convex, finite type domain in $\mathbb{C}^n$, or 
    \item a strongly pseudoconvex domain with smooth boundary in $\mathbb{C}^n$.
\end{enumerate}
Clearly, the unit ball is a model example of a simple domain. We also consider strongly pseudoconvex domains with less boundary smoothness. Theorem A has been extended to strongly pseudoconvex domains with smooth boundary by Wang and Xia in \cite{WX2021}. The sufficiency direction of Theorem B has been extended for $p>1$ to simple domains by Huo, Wick, and the second author in \cites{HWW20201,HWW20202} and to strongly pseudoconvex domains with $C^4$ boundary by Wick and the second author in \cite{WW2020}; this direction of Theorem B when $p=1$ has been extended to strongly pseudoconvex domains with $C^4$ boundary by the authors in \cite{SW2020} and is seen to hold for simple domains following ideas from \cite{M1994}.

Our Theorem \ref{ToeplitzLpCompactness} and Corollary \ref{ContinuousLpCompactness} extend to any simple domain, since our arguments only rely on the good behavior of the Bergman kernel on $\overline{D}\times\overline{D}\setminus\{(z,z): z \in \partial D\}$ which is guaranteed in these settings by \cite{B1987}. For Theorem \ref{CompactnessExtrapolation} to hold in general, we need the boundedness of the Bergman projection with respect to $B_p$ weights and we need $P$ to reproduce $\mathcal{A}^1$ functions. These properties both hold true on, for example, strongly pseudoconvex domains with smooth boundary. The reproducing property can be deduced from an approximation theorem on Bergman spaces together with kernel estimates, while the weighted estimates appear in \cites{HWW20201,HWW20202}. Moreover, Theorem \ref{LpCompactnessCharacterization} and Corollary \ref{LpCompactnessCharacterizationCorollary} also extend to strongly pseudoconvex domains with smooth boundaries since a version of Theorem A is available in this setting. 
Theorem \ref{ToeplitzLpBoundedness}, Theorem \ref{ToeplitzL1Boundedness}, and Corollary \ref{ToeplitzL1BoundednessCorollary}, extend to simple domains, since in this situation the domain can be equipped with a dyadic structure and the Bergman projection is pointwise bounded by a suitable dyadic majorant, see \cite{HWW20202}. 

We now list some open questions connected to this work. 
\begin{open}\label{LpSymbolOpen} Given $p \in (1,\infty)$ and $\sigma \in B_p$, characterize the symbols $u \in L^\infty$ for which $T_u$ is compact on $L^p_\sigma$, where $1<p<\infty$ and $\sigma \in B_p.$ 
\end{open}

\begin{open} \label{ApSymbolOpen} Answer Open Question \ref{LpSymbolOpen} with $L^p_{\sigma}$ 
replaced with  $\mathcal{A}^p_\sigma$.
\end{open}

\begin{open} \label{LpWeightOpen}
Given $p\in(1,\infty)$ and $u \in L^\infty$, characterize the weights $\sigma$ for which $T_u$ acts boundedly on $L^p_\sigma$. Do the same for bounded $T_u$ from $L^p_{\sigma}$ to $L^{p,\infty}_{\sigma}$ for $p\in[1,\infty)$.
\end{open}

\begin{open}\label{WeakTypeDependenceOpen}
    Determine the optimal dependence of $\|P\|_{L^1_{\sigma}\rightarrow L^{1,\infty}_{\sigma}}$ on $[\sigma]_{B_1}$.
\end{open}

Corollary \ref{LpCompactnessCharacterizationCorollary} answers Open Question \ref{LpSymbolOpen} in the case that $u$ is continuous near the boundary. Theorem \ref{LpCompactnessCharacterization} completes the characterization sought in Open Question \ref{ApSymbolOpen} for the subclass of $B_p$ weights satisfying the reverse H\"{o}lder inequality. Theorem \ref{ToeplitzLpBoundedness} and Theorem \ref{ToeplitzL1Boundedness} give sufficient conditions for Open Question \ref{LpWeightOpen} in the cases $p \in (1,\infty)$ and $p=1$, respectively. Corollary \ref{ToeplitzL1BoundednessCorollary} addresses Open Question \ref{WeakTypeDependenceOpen} for the subclass of $B_1$ weights satisfying the reverse H\"older inequality.

The remainder of this paper is organized as follows. In Section \ref{Preliminaries}, we discuss the dyadic structure of $\mathbb{B}_n$, introduce the necessary maximal operators, and prove a reverse H\"older inequality for $B_{\infty}$ weights satisfying \eqref{DyadicWeightRegularity}. In Section \ref{CompactnessSection}, we collect Riesz-Kolmogorov type compactness theorems for weighted Bergman spaces (from which we derive Theorem \ref{ToeplitzLpCompactness} and Corollary \ref{ContinuousLpCompactness}) and then establish Theorem \ref{CompactnessExtrapolation}, Theorem \ref{LpCompactnessCharacterization},  and Corollary \ref{LpCompactnessCharacterizationCorollary}. In Section \ref{BoundednessSection}, we prove Theorem \ref{ToeplitzLpBoundedness}, Theorem \ref{ToeplitzL1Boundedness}, and Corollary \ref{ToeplitzL1BoundednessCorollary}. 
\subsection{Statements and Declarations}
The authors have no competing interests to declare.


\section{Preliminaries}\label{Preliminaries}

We write $A\lesssim B$ if $A\leq CB$ for some $C>0$, and write $A\approx B$ if $A\lesssim B$ and $B \lesssim A$. 

\subsection{Dyadic structure of the unit ball}

We introduce a dyadic structure of ``kubes" on $\mathbb{B}_n$ as in \cites{ARS2006,RTW2017}. 
Fix parameters $\theta_0, \lambda_0 >0$. 
For $k \in \mathbb{N}$, let $\mathbb{S}_{k\theta_0}:= \{z \in \mathbb{B}_n: \beta(0,z)=k\theta_0\}$. We claim that for each $k \in \mathbb{N}$ and $j \in \{1,2,\ldots,J_k\}$, there exist points $w_j^k \in \mathbb{B}_n$ corresponding to Borel sets $Q_j^k\subseteq \mathbb{S}_{k\theta_0}$ containing $w_j^k$, and a constant $C>1$ such that
\begin{align*}
     &\mathbb{S}_{k\theta_0}=\bigcup_{j=1}^{J_k}Q_j^k,\\
     &Q_{j}^k\cap Q_{j'}^k =\emptyset \quad \text{whenever}\,\, j\neq j', \quad\text{and}\\
     &\mathbb{S}_{k\theta_0}\cap B_{\beta}(w_j^k,\lambda_0)\subseteq Q_j^k\subseteq \mathbb{S}_{k\theta_0}\cap B_{\beta}(w_j^k,C\lambda_0).
\end{align*}
For $z \in \mathbb{B}_n$, let $P_{k\theta_0}z$ be the radial projection of $z$ onto $\mathbb{S}_{k\theta_0}$. Define 
$$
     K_1^0:= B_{\beta}(0,\theta_0)
$$
and for $k\in \mathbb{N}$, $j \in \{1,2,\ldots,J_k\}$
$$
     K_j^k:=\{z \in \mathbb{B}_n: k\theta_0 < \beta(0,z)\leq (k+1)\theta_0 \quad \text{and}\quad P_{k\theta_0} z\in Q_j^k\}.
$$

We denote the collection of kubes $K_j^k$ obtained in this construction by $\mathcal{D}_0$. In what follows, we often drop the subscript and superscript labeling a kube if they are not necessary for clarity. Observe that the kubes in $\mathcal{D}_0$ partition $\mathbb{B}_n$. Given a kube $K=K_j^k \in \mathcal{D}_0$, we define its center by $c_K=c_j^k:= P_{(k+\frac{1}{2})\theta_0}w_j^k$. Also, for $K=K_j^k$, we call $k$ the generation of $K$ and denote this by $d(K)$. 
 
We define a tree structure on $\mathcal{D}_0$ in the following way. For $k\ge 1$, we say a kube $K_{j_1}^{k+1}$ is a child of $K_{j_2}^{k}$ if $P_{k \theta_0}c_{j_1}^{k+1} \in Q_{j_2}^{k}$ and we declare every kube in $\{K_j^1\}_{j=1}^{J_1}$ to be a child of $K_1^0$. More generally, we say $K_{j_1}^{k_1}$ is a descendant of $K_{j_2}^{k_2}$ and write $K_{j_1}^{k_1} \prec K_{j_2}^{k_2}$ if $k_1>k_2$ and $P_{k_2 \theta_0}c_{j_1}^{k_1} \in Q_{j_2}^{k_2}$. We denote the collection of children of $K \in \mathcal{D}_0$ by $\text{ch}(K)$ and define $\mathcal{D}(K):= \{K\}\cup \{K' \in \mathcal{D}_0: K' \prec K\}$. Given a fixed kube $K\in\mathcal{D}_0$, we define the dyadic tent over the kube, $\widehat{K}$, to be the following set:
$$
    \widehat{K}:= \bigcup_{K' \in \mathcal{D}(K)}  K'.
$$
Notice that by construction,  given two tents $\widehat{K}$ and $\widehat{K'}$ associated to kubes in $\mathcal{D}_0$, either the two sets are disjoint or one is contained in the other. In addition, \cite{RTW2017}*{Lemma 1} asserts that
\begin{align}\label{KubeTentEquivalence}
    |K| \approx |\widehat{K}| \approx (1-|c_K|^2)^{n+1}\approx e^{-2d(K)\theta_0(n+1)}
\end{align}
for any $K \in \mathcal{D}_0$. We explicitly write
\begin{align}\label{SparseCondition}
    |K| \ge (1-\rho_{0})|\widehat{K}|,
\end{align}
where $\rho_{0}\in(0,1)$, and therefore 
\begin{align*}
    \bigg|\bigcup_{K' \in \text{ch}(K)}\widehat{K'}\bigg|\leq \rho_{0}|\widehat{K}|.
\end{align*}

We now prove the important dyadic doubling condition.
\begin{lemma} \label{DyadicDoubling} There exists $\alpha_{0}>1$ such that if $K \in \mathcal{D}_0$ and $K' \in \text{ch}(K)$, then 
$$ 
    |\widehat{K}| \leq \alpha_{0} |\widehat{K'}|.
$$
\end{lemma}
\begin{proof}
Note that by definition, for any kube $K\in\mathcal{D}_0$ with $d(K)=k$, we have
$$
    \beta(0,c_K)= \frac{1}{2} \log \left(\frac{1+|c_K|}{1-|c_K|}\right)=\left(k+\frac{1}{2}\right)\theta_0.
$$
We thus obtain
$$
    |c_K| = \frac{e^{2(k+\frac{1}{2}) \theta_0}-1}{e^{2(k+\frac{1}{2}) \theta_0}+1},
$$
and moreover
\begin{align*}
|\widehat{K}|&  \approx (1-|c_K|^2)^{n+1} \\
& \approx (1-|c_K|)^{n+1}\\
& =\left(\frac{2}{e^{2(k+\frac{1}{2}) \theta_0}+1}\right)^{n+1} \\
& \approx e^{-2(k+\frac{1}{2})\theta_0(n+1)}.
\end{align*}
Therefore, we estimate
\begin{align*}
|\widehat{K}| & \approx e^{-2(k+\frac{1}{2})\theta_0(n+1)}\\
& = e^{2\theta_0(n+1)} e^{-2(k+1+\frac{1}{2})\theta_0(n+1)}\\
& \approx e^{2 \theta_0(n+1)} |\widehat{K'}|,
\end{align*}
which proves the result.
\end{proof}

It is a fact that every Carleson tent can be approximated by a dyadic tent, which is made precise by the following extension of \cite{RTW2017}*{Lemma 3}.
\begin{lemma}\label{TentApproximation}
 There exists a finite collection of dyadic structures $\{ \mathcal{D}_{\ell}\}_{\ell=1}^{M}$ such that for any $z \in \mathbb{B}_n$, there exists $K \in \bigcup_{\ell=1}^{M} \mathcal{D}_{\ell}$ such that $\widehat{K} \supseteq \mathcal{T}_z$ and $|\widehat{K}| \approx |\mathcal{T}_z|$. Moreover, for any $K' \in \bigcup_{\ell=1}^{M}\mathcal{D}_{\ell}$, there exists $z' \in \mathbb{B}_n$ such that $\mathcal{T}_{z'}\supseteq\widehat{K'}$ and $ |\mathcal{T}_{z'}|\approx |\widehat{K'}|$. 
\end{lemma}

\begin{proof}
The first assertion follows from \cite{RTW2017}*{Lemma 3}.

Note that any dyadic tent $\widehat{K}_{j}^{k} \in \mathcal{D}_{\ell}$ can be formed by projecting some subset $D_j^k \subseteq \partial \mathbb{B}_n$  onto spheres in the Bergman metric (see \cite{RTW2017}*{Lemma 3}). The set $D_j^k$ is part of a dyadic system, say of caliber $\delta$, on $\partial\mathbb{B}_n$ with respect to the pseudo-metric $\rho(z,w)= |1-z\overline{w}|$. In particular, there exist $z_j^k \in \partial \mathbb{B}_n$ and constants $c,C>0$ such that
$$
    D(z_j^k, c\delta^k) \subseteq D_j^k \subseteq D(z_j^k, C\delta^k),
$$
where $D(z,r):= \{w \in \partial \mathbb{B}_n: \rho(z,w)<r\}$. Without loss of generality, we may assume $C\geq 2$. Note that the parameter $\theta$ is chosen so that $\delta=e^{-2 \theta}$. Clearly, there exists $z' \in \mathbb{B}_n$ such that $1-|z'| = C\delta^k $ and $P_{\infty}z'=z_j^k,$ where $P_{\infty}$ denotes projection to the boundary. Consider the Carleson tent $\mathcal{T}_{|z'|z'}.$ There holds
$$|\mathcal{T}_{|z'|z'}| \approx (1-|z'|^2)^{n+1} \approx \delta^{k(n+1)}=e^{-2k\theta(n+1)} \approx |\widehat{K}_j^k| ,$$
so it remains to prove $\widehat{K}_j^k \subseteq \mathcal{T}_{|z'|z'}.$ 

Take $w \in \widehat{K}_j^k$ and let $w'=\frac{w}{|w|} \in \partial \mathbb{B}_n.$ It is clear from construction that $w' \in D_j^k \subseteq D(z_j^k, C\delta^k)= D(P_{\infty}z',1-|z'|).$ We claim that this implies $w \in \mathcal{T}_{|z'|z'}.$ We have that
$$
    \left| 1-\frac{\overline{w}}{|w|}\frac{z'}{|z'|}\right|< 1-|z'|,
$$ 
and we must show 
$$
    \left| 1-\overline{w}\frac{z'}{|z'|}\right|< 1-|z'|^2.
$$ 
Note that by definition, $|w| \geq \frac{e^{2 k \theta}-1}{e^{2k \theta}+1}$, and so
$$
    1-|z'|  = C e^{-2k\theta} \geq \frac{2}{e^{2k\theta}+1} \geq 1-|w|,
$$
or equivalently, $|w| \geq |z'|.$ We use the fact that $|w| \geq |z'|$ to compute
\begin{align*}
\left| 1-\overline{w}\frac{z'}{|z'|}\right| & =  \left| 1-|w|+|w|-\overline{w}\frac{z'}{|z'|}\right|\\
& \leq 1-|w|+ |w|\left| 1-\frac{\overline{w}}{|w|}\frac{z'}{|z'|}\right|\\
& < 1-|w| + |w|(1-|z'|)\\
& = 1-|w||z'|\\
& \leq 1-|z'|^2,
\end{align*}
\end{proof}

The collection $\{\mathcal{D}_{\ell}\}_{\ell=1}^{M}$ remains fixed throughout this paper, and we put $\mathcal{D}:=\bigcup_{\ell=1}^{M}\mathcal{D}_{\ell}$. To ensure our estimates hold for any of these dyadic structures, we set 
$$
    \rho:=\max_{1\leq \ell\leq M}\rho_{\ell}\quad \text{and}\quad \alpha:=\max_{1\leq \ell\leq M} \alpha_{\ell}.
$$ 
It is a fact (see \cite{ARS2006}*{Lemma 2.8}) that there exist independent $C_{1}, C_{2}>0$ such that
\begin{align*}
B_{\beta}(c_K,C_{1}) \subseteq K \subseteq B_{\beta}(c_K,C_{2})
\end{align*}
for any $K \in\mathcal{D}$. Therefore, 
it is clear that \eqref{WeightRegularity} implies \eqref{DyadicWeightRegularity}.

Lemma \ref{TentApproximation} implies that the $B_p$ and $B_p^d$ characteristics may be used interchangeably when working with B\'ekoll\`e-Bonami weights, as is stated in the following proposition.
\begin{prop}
If $p \in [1,\infty)$ and $\sigma \in B_p$, then
$$
    [\sigma]_{B_p} \lesssim [\sigma]_{B_p^d} \lesssim [\sigma]_{B_p}.
$$
\end{prop}

\begin{proof}
Fix $z \in \mathbb{B}_n$. By Lemma \ref{TentApproximation}, there exists $K \in \mathcal{D}$ with $\widehat{K} \supseteq \mathcal{T}_z$ and $|\widehat{K}|\approx|\mathcal{T}_z|$. Then 
$$
    \langle \sigma \rangle_{\mathcal{T}_z}\langle \sigma^{1-p'}\rangle_{\mathcal{T}_z}^{p-1} \lesssim \langle \sigma \rangle_{\widehat{K}}\langle \sigma^{1-p'}\rangle_{\widehat{K}}^{p-1}
$$
for $1<p<\infty$, and
$$
    \langle \sigma \rangle_{\mathcal{T}_z}\|\sigma^{-1}\|_{L^{\infty}(\mathcal{T}_z)} \lesssim \langle \sigma \rangle_{\widehat{K}}\| \sigma^{-1}\|_{L^{\infty}(\widehat{K})}.
$$ 
Taking the supremum over $z \in \mathbb{B}_n$ yields $[\sigma]_{B_p}\lesssim [\sigma]_{B_p^d}$. The same argument with the roles of $\widehat{K}$ and $\mathcal{T}_z$ reversed is also justified by Lemma \ref{TentApproximation} and implies $[\sigma]_{B_p^d}\lesssim [\sigma]_{B_p}$.
\end{proof}

\subsection{Maximal operators and Young functions}

Given $K \in \mathcal{D}$, we define the following localized maximal operator:
$$
     M_{K}f(z):= \sup_{K' \in \mathcal{D}(K)} \langle |f| \rangle_{\widehat{K'}}\chi_{\widehat{K'}}(z).
$$ 
We have the following important property for weights satisfying \eqref{DyadicWeightRegularity}.
\begin{lemma} \label{MaximalDyadicControl}
If $\sigma$ is a weight satisfying \eqref{DyadicWeightRegularity}, $K \in \mathcal{D}$, and $z \in \widehat{K}$, then
$$
    \sigma(z) \leq \frac{c_{\sigma}}{1-\rho}  M_{K} \sigma(z).
$$
\end{lemma}
\begin{proof}
Since $z \in \widehat{K}$, there exists $K' \in \mathcal{D}(K)$ with $z \in K'$. Using \eqref{DyadicWeightRegularity} and \eqref{SparseCondition}, we conclude 
$$
    \sigma(z) \leq c_{\sigma} \langle \sigma \rangle_{K'} \leq  \frac{c_{\sigma}}{1-\rho} \langle \sigma \rangle_{\widehat{K'}} \leq \frac{c_{\sigma}}{1-\rho} M_{K} \sigma(z).
$$
\end{proof} 


The maximal operator associated to $u \in L^{\infty}$ is given by 
$$
    M_{u}f(z):= \sup_{K \in \mathcal{D}}\|u\|_{L^{\infty}(\widehat{K})}\langle |f|\rangle_{\widehat{K}}\chi_{\widehat{K}}(z).
$$ 
Notice that if $\sigma \in uB_1$ then $M_{u}\sigma (z)\lesssim [\sigma]_{uB_1}\sigma(z)$. We have the following weak-type inequality for $M_{u}$ that will be used in proving Theorem \ref{ToeplitzL1Boundedness}.
\begin{lemma}\label{FeffermanStein} If $u \in L^{\infty}$ and $\sigma$ is a weight, then
$$
    \|M_{u}f\|_{L_{\sigma}^{1,\infty}} \lesssim \|f\|_{L_{M_{u}\sigma}^{1}}
$$
for all $f \in L_{M_{u}\sigma}^1$.
\end{lemma}
\begin{proof}
Let $\lambda>0$. 
Let $\Omega_{\lambda,\ell}$ be the collection of $K \in \mathcal{D}_{\ell}$ such that $\widehat{K}$ is maximal with respect to inclusion in $\{\widehat{K}: K \in \mathcal{D}_{\ell} \,\,\,\text{and}\,\,\, \|u\|_{L^{\infty}(\widehat{K})}\langle|f|\rangle_{\widehat{K}} > \lambda\}$ and put $\Omega_{\lambda}=\bigcup_{\ell=1}^{M}\Omega_{\lambda,\ell}$. Notice that $\widehat{K}\cap\widehat{K'}=\emptyset$ for distinct $K,K' \in \Omega_{\lambda,\ell}$ and that $\{M_{u}f>\lambda\} = \bigcup_{K \in \Omega_{\lambda}}\widehat{K}$. Therefore
\begin{align*}
    \sigma(\{M_{u}f>\lambda\})&\leq\sum_{K\in \Omega_{\lambda}}\sigma(\widehat{K})\\
    &<\frac{1}{\lambda}\sum_{K\in \Omega_{\lambda}}\|u\|_{L^{\infty}(\widehat{K})}\langle |f|\rangle_{\widehat{K}}\sigma(\widehat{K})\\
    &=\frac{1}{\lambda}\sum_{K\in \Omega_{\lambda}}\int_{\widehat{K}}|f(z)|\|u\|_{L^{\infty}(\widehat{K})}\langle\sigma\rangle_{\widehat{K}}\,dV(z)\\
    &\leq \frac{1}{\lambda}\sum_{\ell=1}^{M}\sum_{K\in\Omega_{\lambda,\ell}}\int_{\widehat{K}}|f(z)|M_{u}\sigma(z)\,dV(z)\\
    &\leq \frac{M}{\lambda} \|f\|_{L_{M_{u}\sigma}^1}.
\end{align*}
\end{proof}

We will use the following maximal operator adapted to a weight $\sigma$:
$$
    M_{\sigma}f(z):=\sup_{K \in \mathcal{D}} \frac{1}{\sigma(\widehat{K})}\left(\int_{\widehat{K}}|f|\sigma\,dV\right) \chi_{\widehat{K}}(z).
$$
The following lemma is well-known, see \cite{RTW2017}*{Lemma 4} for example. 
\begin{lemma}\label{WeightedMaximal}
If $\sigma$ is a weight and $p \in (1,\infty)$, then $M_{\sigma}$ is bounded on $L^p_{\sigma}$. Moreover, $\|M_{\sigma}\|_{L^p_{\sigma}\rightarrow L^p_{\sigma}}$ does not depend on $\sigma$.
\end{lemma}

We call a function $\Phi:[0,\infty)\rightarrow [0,\infty)$ a Young function if $\Phi$ is continuous, convex, increasing, and satisfies $\Phi(0)=0$. This implies that the inverse function $\Phi^{-1}$ defined on $(0,\infty)$ exists. We define the maximal operator associated to $u \in L^{\infty}$ and $\Phi$ by
$$
    M_{u,\Phi}f(z):=\sup_{K\in\mathcal{D}} \|u\|_{L^{\infty}(\widehat{K})}\|f\|_{\Phi,\widehat{K}}\chi_{\widehat{K}}(z),
$$
where
$$
    \|f\|_{\Phi,\widehat{K}}:=\inf\left\{\lambda >0 : \frac{1}{|\widehat{K}|}\int_{\widehat{K}} \Phi\left(\frac{|f|}{\lambda}\right)\,dV \leq 1\right\}.
$$
Note that 
\begin{align}\label{MaximalPointwise}
M_{u}f(z) \lesssim M_{u,\Phi}f(z)
\end{align}
for any Young function $\Phi$ we consider (see \cite{CUMP2011}*{Page 99}).

Given a Young function $\Phi$, its complementary Young function, $\Psi$, is given by 
$$
    \Psi(t):=\sup_{s>0}\{st-\Phi(s)\}.
$$
Assuming $\Phi$ satisfies $\lim_{t\rightarrow \infty}\frac{\Phi(t)}{t}=\infty$, $\Psi$ is a well-defined Young function taking values in $[0,\infty)$. Moreover, we have the following H\"older inequality:
\begin{equation}\label{OrliczHolder}
    \langle fg \rangle_{A}\lesssim \|f\|_{\Phi,A}\|g\|_{\Psi,A}
\end{equation}
for $A\subseteq \mathbb{B}_n$, complementary Young functions $\Phi$ and $\Psi$, and $f, g$ with $\|f\|_{\Phi,A},\|g\|_{\Psi,A}<\infty$. The following is an easy consequence of \eqref{OrliczHolder}.
\begin{lemma}\label{OrliczHolderApplied}
If $A \subseteq \mathbb{B}_n$, $A' \subseteq A$, and $\sigma: A' \rightarrow [0,\infty)$, then
$$
    \langle \sigma \rangle_A \lesssim \frac{\|\sigma\|_{\Phi,A}}{\Psi^{-1}(|A|/|A'|)}.
$$
\end{lemma}


\subsection{Reverse H\"older property of $B_{\infty}$ weights satisfying \eqref{DyadicWeightRegularity}}

The following is an analogue of \cite{HPR2012}*{Lemma 2.2} and is key in proving our improved reverse H\"older inequality.
\begin{lemma} \label{ReverseHolderLemma}
Let $K \in \mathcal{D}$. If $\sigma \in B_\infty$ and $0<\delta< \frac{1}{2\alpha[\sigma]_{B_\infty}}$, then
$$
    \langle (M_{K}\sigma)^{1+\delta} \rangle_{\widehat{K}}\leq 2 [\sigma]_{B_\infty} \langle \sigma \rangle_{\widehat{K}}^{1+\delta}.
$$
\begin{proof}
Without loss of generality, we may assume $\sigma$ is supported on $\widehat{K}$. For each $\lambda>0$, define $\Omega_\lambda=\{ M_{K}\sigma>\lambda\}$. Note that $\Omega_\lambda= \bigcup_{j} \widehat{K}_{j,\lambda}$, a disjoint union where the $\widehat{K}_{j,\lambda}$ are maximal dyadic tents contained in $\widehat{K}$ with $\langle \sigma\rangle_{\widehat{K}_{j, \lambda}}>\lambda$. 
Using the distribution function, this decomposition of $\Omega_\lambda$, the definition of $B_{\infty}$, and the maximality of the $\widehat{K}_{j,\lambda}$:
\begin{align*}
\int_{\widehat{K}}( M_{K} \sigma)^{1+\delta} \, dV & = \int_{0}^{\infty} \delta \lambda^{\delta-1} M_{K} \sigma(\Omega_\lambda) \mathop{d \lambda} \\
& = \int_{0}^{\langle \sigma \rangle_{\widehat{K}}} \delta \lambda^{\delta-1} M_{K}\sigma(\widehat{K}) \mathop{d \lambda}+ 
\int_{\langle \sigma \rangle_{\widehat{K}}}^{\infty} \delta \lambda^{\delta-1} M_{K}\sigma(\Omega_\lambda) \mathop{d \lambda}   \\
& = \langle \sigma \rangle_{\widehat{K}}^{\delta} M_{K}\sigma(\widehat{K})+ 
\int_{\langle \sigma \rangle_{\widehat{K}}}^{\infty} \delta \lambda^{\delta-1} \sum_{j} M_{K}\sigma(\widehat{K}_{j,\lambda}) \mathop{d \lambda}\\
& \leq \langle \sigma \rangle_{\widehat{K}}^{\delta} M_{K}\sigma(\widehat{K})+ [\sigma]_{B_\infty}\int_{\langle \sigma \rangle_{\widehat{K}}}^{\infty} \delta \lambda^{\delta-1} \sum_{j} \sigma(\widehat{K}_{j,\lambda}) \mathop{d \lambda}.
\end{align*}

Note that by maximality, each $\widehat{K}_{j,\lambda}$ is contained in a unique dyadic parent $\widehat{K}_{j,\lambda}'$ with the property that $\langle \sigma\rangle_{\widehat{K}_{j,\lambda}'} \leq \lambda.$ Then, using Lemma \ref{DyadicDoubling}, we have
$$
    \langle \sigma\rangle_{\widehat{K}_{j,\lambda}} \leq \alpha \langle \sigma\rangle_{\widehat{K}_{j,\lambda}'} \leq \alpha \lambda.
$$
Using the above inequality, the disjointness of the $\widehat{K}_{j,\lambda}$, and the distribution function, we continue estimating above by
\begin{align*}
& \langle \sigma \rangle_{\widehat{K}}^{\delta} M_{K}\sigma(\widehat{K})+ [\sigma]_{B_\infty} \alpha \delta \int_{\langle \sigma \rangle_{\widehat{K}}}^{\infty} \lambda^{\delta} \sum_{j} |\widehat{K}_{j,\lambda}| \mathop{d \lambda}\\
 & = \langle \sigma \rangle_{\widehat{K}}^{\delta} M_{K}\sigma(\widehat{K})+ [\sigma]_{B_\infty} \alpha \frac{\delta}{1+\delta} \int_{\langle \sigma \rangle_{\widehat{K}}}^{\infty} (1+\delta)\lambda^{\delta} |\Omega_\lambda| \mathop{d \lambda} \\
 & \leq \langle \sigma \rangle_{\widehat{K}}^{\delta} M_{K}\sigma(\widehat{K})+  [\sigma]_{B_\infty} \alpha \frac{\delta}{1+\delta} \int_{\widehat{K}}(M_{K}\sigma)^{1+\delta} \, dV.
\end{align*}
Then, dividing both sides of the inequality by $|\widehat{K}|$ and again using the $B_\infty$ condition, we get
$$
    \langle (M_{K}\sigma)^{1+\delta} \rangle_{\widehat{K}} \leq [\sigma]_{B_\infty} \langle \sigma \rangle_{\widehat{K}}^{1+\delta}+ \frac{\alpha \delta [\sigma]_{B_\infty}}{1+\delta} \langle (M_{K}\sigma)^{1+\delta} \rangle_{\widehat{K}}.
$$ 
 The observation that $\frac{\alpha \delta [\sigma]_{B_\infty}}{1+\delta} \leq \frac{1}{2}$ completes the proof, as the second term can be subtracted over to the other side. 
\end{proof}
\end{lemma}

The following is an improved Reverse H\"older inequality for $B_{\infty}$ weights satisfying \eqref{DyadicWeightRegularity}.
\begin{thm}\label{ReverseHolder}
If $\sigma \in B_{\infty}$ satisfies \eqref{DyadicWeightRegularity} and $1<r \leq 1+ \frac{1}{2\alpha [\sigma]_{B_\infty}}$, then $\sigma \in \text{RH}_{r}$ with
$$
[\sigma]_{\text{RH}_r}\leq \frac{2}{1-\rho}c_{\sigma}^{\frac{r-1}{r}}.
$$
\end{thm}
\noindent Theorem \ref{ReverseHolder} extends to multiple dimensions and sharpens a result contained in \cite{APR2019}*{Theorem 1.7}. In particular, the proof in \cite{APR2019} shows that $[\sigma]_{\text{RH}_r}$ depends at worst exponentially on $[\sigma]_{B_{\infty}}$, whereas our constant is independent of $\sigma$ if $c_\sigma$ is bounded. Our proof follows an argument from \cite{HPR2012} of a sharp reverse H\"{o}lder inequality for $A_\infty$ weights on Euclidean space.

\begin{proof}[Proof of Theorem \ref{ReverseHolder}]
Write $r=1+\delta$, where $0<\delta \leq \frac{1}{2\alpha [\sigma]_{B_\infty}}.$ Fix $K \in \mathcal{D}$. We begin by applying Lemma \ref{MaximalDyadicControl}:
$$
    \int_{\widehat{K}} \sigma^{1+\delta} \, dV \leq \left(\frac{c_{\sigma}}{1-\rho}\right)^{\delta} \int_{\widehat{K}} (M_{K} \sigma)^{\delta} \sigma \, dV.
$$

We use the notation of Lemma \ref{ReverseHolderLemma}, letting $\Omega_{\lambda}=\{M_{K}\sigma>\lambda\}=\bigcup_{j}\widehat{K}_{j,\lambda}$. We continue estimating using arguments similar to those in Lemma \ref{ReverseHolderLemma}:
\begin{align*}
\int_{\widehat{K}}(M_{K}\sigma)^{\delta} \sigma \, dV & = \int_{0}^{\infty} \delta \lambda^{\delta-1} \sigma(\Omega_\lambda) \mathop{d \lambda} \\
& = \int_{0}^{\langle \sigma \rangle_{\widehat{K}}} \delta \lambda^{\delta-1} \sigma(\widehat{K}) \mathop{d \lambda}+ 
\int_{\langle \sigma \rangle_{\widehat{K}}}^{\infty} \delta \lambda^{\delta-1} \sigma(\Omega_\lambda) \mathop{d \lambda}   \\
& = \langle \sigma \rangle_{\widehat{K}}^{\delta} \sigma(\widehat{K})+ 
\int_{\langle \sigma \rangle_{\widehat{K}}}^{\infty} \delta \lambda^{\delta-1} \sum_{j} \sigma(\widehat{K}_{j,\lambda}) \mathop{d \lambda}\\
& \leq \langle \sigma \rangle_{\widehat{K}}^{\delta} \sigma(\widehat{K})+ 
\int_{\langle \sigma \rangle_{\widehat{K}}}^{\infty} \alpha \delta \lambda^{\delta} \sum_{j} |\widehat{K}_{j,\lambda}| \mathop{d \lambda}\\
& = \langle \sigma \rangle_{\widehat{K}}^{\delta} \sigma(\widehat{K})+ 
\frac{\alpha \delta}{1+\delta} \int_{\langle \sigma \rangle_{\widehat{K}}}^{\infty}  (1+\delta) \lambda^{\delta} |\Omega_\lambda| \mathop{d \lambda}\\
& \leq \langle \sigma \rangle_{\widehat{K}}^{\delta} \sigma(\widehat{K})+ 
\frac{\alpha \delta}{1+\delta} \int_{\widehat{K}} (M_{K}\sigma)^{1+\delta} \, dV.
\end{align*}
Dividing both sides by $|\widehat{K}|$ and using Lemma \ref{ReverseHolderLemma}, we obtain
\begin{align*}
    \langle \sigma^{1+\delta}\rangle_{\widehat{K}}  &\leq \left(\frac{c_{\sigma}}{1-\rho}\right)^{\delta}\left( 1+ \frac{2\alpha [\sigma]_{B_\infty}\delta}{1+\delta} \right) \langle \sigma \rangle_{\widehat{K}}^{1+\delta}\\
    &\leq 2\left(\frac{c_{\sigma}}{1-\rho}\right)^{\delta}\langle \sigma \rangle_{\widehat{K}}^{1+\delta}.
\end{align*}
Taking a supremum over $K\in\mathcal{D}$ gives
$$
    [\sigma]_{\text{RH}_r}\leq 2^{1/r}\left(\frac{c_{\sigma}}{1-\rho}\right)^{\frac{r-1}{r}}\leq \frac{2}{1-\rho}c_{\sigma}^{\frac{r-1}{r}}.
$$
\end{proof}

\section{Weighted compactness}\label{CompactnessSection}

\subsection{Weighted compactness for general $B_p$ weights}\label{LpCompactnessSubsection}
In proving Theorem \ref{ToeplitzLpCompactness}, we use a version of the Riesz-Kolmorogov characterization of precompact sets adapted to weighted Bergman spaces. We use the notation $r \mathbb{B}_n$ to represent the set
$$
    r\mathbb{B}_n:=\{z \in \mathbb{C}_n:|z|<r\}.
$$ 
For $r \in (0,1)$, we write $r\mathbb{B}_n^c$ to refer to the complemented subset $\mathbb{B}_n \setminus r \mathbb{B}_n.$

The following theorem is contained in \cite{MSWW2021}*{Corollary 1.10}. 
\begin{thm}\label{RieszKolmogorov}
Let $p \in (1,\infty)$ and suppose that $\sigma^{1-p'} \in L^1$. 
A subset $\mathcal{F} \subseteq \mathcal{A}^p_{\sigma}$ is precompact if and only if
$$
    \lim_{r\rightarrow 1^-}\sup_{f \in \mathcal{F}}\int_{r\mathbb{B}_n^c}|f|^p \sigma \, dV=0.
$$
\end{thm}

We also have a Riesz-Kolmogorov type result for weak-type spaces and general weights.
\begin{thm}\label{RKWeakType}
Let $\sigma$ be a weight and $\mathcal{F} \subseteq \mathcal{A}^{1,\infty}_{\sigma} $ be a bounded set. If the functions in $\mathcal{F}$ are uniformly bounded on compact subsets of $\mathbb{B}_n$ and
$$
    \lim_{r \rightarrow 1^{-}} \sup_{f \in \mathcal{F}} \sup_{\lambda>0} \lambda \sigma\left(\{z:|f(z)|>\lambda\} \cap r\mathbb{B}_n^c\right)=0,
$$
then $\mathcal{F}$ is precompact in $L^{1,\infty}_{\sigma}$.
\end{thm}

\begin{proof} Let $\{f_j\}_{j=1}^{\infty}$ be any sequence in $\mathcal{F}.$ Note that by Montel's theorem, by passing to a subsequence if necessary, we may assume that $\{f_j\}_{j=1}^{\infty}$ converges uniformly on compact sets to a holomorphic function $g$. We first show that $g \in L^{1,\infty}_{\sigma}.$ It is straightforward to verify that for any fixed $r \in (0,1)$, uniform convergence implies

\begin{equation}\sup_{\lambda>0} \lambda \sigma\left(\{z:|g(z)|>\lambda\}\cap r\mathbb{B}_n \right) \leq 2\sup_{f \in  \mathcal{F}}\|f\|_{L^{1,\infty}_{\sigma}}. \label{1}\end{equation}
 In particular, we have, for any $j$,
\begin{align*}& \sup_{\lambda>0} \lambda \sigma\left(\{z:|g(z)|>\lambda\}\cap r\mathbb{B}_n \right) \\
&  \quad\leq 2\sup_{\lambda>0} \lambda \sigma\left(\{z:|f_j(z)|>\lambda\}\cap r\mathbb{B}_n \right) + 2\sup_{\lambda>0} \lambda \sigma\left(\{z:|g(z)-f_j(z)|>\lambda\}\cap r\mathbb{B}_n \right).
\end{align*}
The second term can be dominated by
$$2 \int_{r \mathbb{B}_n} |g-f_j|\sigma \, dV,$$
which goes to $0$ by uniform convergence. Taking a $\limsup$ in $j$ on both sides, we obtain \eqref{1}. Then take the supremum over all $r$ and interchange the two suprema to get that $g \in L^{1,\infty}_{\sigma}.$ 

It remains to show that 
$$
    \lim_{j \rightarrow \infty} \sup_{\lambda>0} \lambda \sigma\left(\{z:|g(z)-f_j(z)|>\lambda\} \right)=0.
$$
Note that the argument above shows that for any $r \in (0,1)$ and $r' \in (r,1),$ we have
$$
    \sup_{\lambda>0} \lambda \sigma\left(\{z:|g(z)|>\lambda\} \cap (r'\mathbb{B}_n \setminus r \mathbb{B}_n) \right)\leq  2 \sup_{j} \sup_{\lambda>0} \lambda \sigma\left(\{z:|f_j(z)|>\lambda\}\cap (r'\mathbb{B}_n \setminus r \mathbb{B}_n)  \right).
$$
Taking a supremum over such $r'$ yields
\begin{equation}
  \sup_{\lambda>0} \lambda \sigma\left(\{z:|g(z)|>\lambda\} \cap  r \mathbb{B}_n^c \right)\leq  2 \sup_{j} \sup_{\lambda>0} \lambda \sigma\left(\{z:|f_j(z)|>\lambda\}\cap r \mathbb{B}_n^c \right). 
\label{2} \end{equation}
Given $\varepsilon>0$, the hypothesis and \eqref{2} imply we can choose $r$ sufficiently close to $1$ so that

$$ \sup_{j} \sup_{\lambda>0} \lambda \sigma\left(\{z:|f_j(z)|>\lambda\} \cap r\mathbb{B}_n^c\right)<\frac{\varepsilon}{2}$$
and
$$  \sup_{\lambda>0} \lambda \sigma\left(\{z:|g(z)|>\lambda\} \cap r\mathbb{B}_n^c\right)< \frac{\varepsilon}{2}.$$
Note that 
\begin{align*}
& \sup_{\lambda>0} \lambda \sigma\left(\{z:|g(z)-f_j(z)|>\lambda\} \right)\\
& \quad\leq \sup_{\lambda>0} \lambda \sigma\left(\{z:|g(z)-f_j(z)|>\lambda\} \cap r\mathbb{B}_n\right)+ \sup_{\lambda>0} \lambda \sigma\left(\{z:|g(z)-f_j(z)|>\lambda\} \cap r\mathbb{B}_n^c\right)\\
& \quad\leq \sup_{\lambda>0} \lambda \sigma\left(\{z:|g(z)-f_j(z)|>\lambda\} \cap r\mathbb{B}_n\right)+ 2\sup_{\lambda>0} \lambda \sigma\left(\{z:|g(z)|>\lambda\} \cap r\mathbb{B}_n^c\right)\\
& \quad\quad+ 2\sup_{\lambda>0} \lambda \sigma\left(\{z:|f_j(z)|>\lambda\} \cap r\mathbb{B}_n^c\right)\\
& \quad< \sup_{\lambda>0} \lambda \sigma\left(\{z:|g(z)-f_j(z)|>\lambda\} \cap r\mathbb{B}_n\right) + \varepsilon.
\end{align*}
Taking a limsup as $j \rightarrow \infty$ on both sides and using the same reasoning as before for the convergence on compact subsets, we obtain

$$\limsup_{j \rightarrow \infty} \sup_{\lambda>0} \lambda \sigma\left(\{z:|g(z)-f_j(z)|>\lambda\} \right)< \varepsilon.$$
Since $\varepsilon$ was arbitrary, the result follows. 
\end{proof}

Theorem \ref{RieszKolmogorov} and Theorem \ref{RKWeakType} admit the following corollary, which is our primary use.
\begin{cor}\label{OperatorCompactnessTest}
Let $p \in [1,\infty)$, $\sigma \in B_p$, and $T$ be a bounded operator from $L^p_{\sigma}$ to $\mathcal{A}^p_{\sigma}$. Then $T$ is compact on $L^p_{\sigma}$ if and only if
$$
    \lim_{r \rightarrow 1^{-}}\sup_{\substack{f \in L^p_{\sigma}\\ \|f\|_{L^p_{\sigma}}\leq 1}}\int_{r\mathbb{B}_n^c} |Tf|^p \sigma \,dV=0.
$$
Moreover, if $\sigma$ is a weight and $T$ is a bounded operator from $L^1_{\sigma}$ to $\mathcal{A}^{1,\infty}_{\sigma}$, the elements of $\{Tf : \|f\|_{L^{1}_{\sigma}}\leq 1\}$ are uniformly bounded on compact subsets of $\mathbb{B}_n$, and 
$$
    \limsup_{r\rightarrow 1^-}\sup_{\substack{f \in L^1_{\sigma} \\ \|f\|_{L^1_{\sigma}}\leq 1}}\sup_{\lambda>0}\lambda \sigma\left(\{z:|Tf(z)|>\lambda\} \cap r\mathbb{B}_n^c\right)=0,
$$
then $T$ is compact from $L^1_{\sigma}$ to $L^{1,\infty}_{\sigma}$.
\end{cor}

\begin{proof}[Proof of Theorem \ref{ToeplitzLpCompactness}] Note that if $p \in (1,\infty)$ and $\sigma \in B_p$, then $L^p_{\sigma}\subseteq L^1$ by H\"{o}lder's inequality, and if $\sigma \in B_1$, then $L^1_\sigma \subseteq L^1$ since $B_1$ weights are bounded from below by a positive constant. Since the Bergman projection 
produces holomorphic functions from $L^1$ data, our Toeplitz operators actually map $L^p_\sigma$ to $\mathcal{A}^p_\sigma$ for $p \in (1,\infty)$ and $\sigma \in B_p$ and map $L^1_{\sigma}$ to $\mathcal{A}^{1,\infty}_{\sigma}$ for $\sigma \in B_1$. Therefore, we can apply Corollary \ref{OperatorCompactnessTest} to the operator $T_u$. If the symbol $u$ vanishes on $\partial \mathbb{B}_n$, a routine argument by splitting $f \in L^p_{\sigma}$ as $f=f \chi_{r \mathbb{B}_n}+f\chi_{r \mathbb{B}_n^c}$ for appropriately chosen $r \in (0,1)$ shows that the corresponding uniform decay condition is satisfied. By Corollary \ref{OperatorCompactnessTest}, Theorem \ref{ToeplitzLpCompactness} is proven.
\end{proof}

The following lemma is well-known. The proof in the case of the unit disk in $\mathbb{C}$ given in \cite{Zhu1}*{Proposition 6.14} readily extends to $\mathbb{B}_n$.
\begin{lemma}\label{BoundaryValues} 
If $u \in L^{\infty}$ is continuous on $\overline{\mathbb{B}}_n\setminus K$ for some compact $K \subseteq \mathbb{B}_n$ and $z_0\in\partial\mathbb{B}_n$, then
$$
    \lim_{z\rightarrow z_0}\widetilde{T_u}(z) = u(z_0).
$$
\end{lemma}

We say that two Banach (quasi-Banach) spaces $\mathcal{X}_0$ and $\mathcal{X}_1$ form a Banach (quasi-Banach) couple if there is a Hausdorff topological linear space 
containing each $\mathcal{X}_j$. For $\theta \in (0,1)$ and $q \in (0,\infty]$, let $(\mathcal{X}_0,\mathcal{Y}_1)_{\theta,q}$ denote the real interpolation space, see \cite{BerghLofstrom}*{Chapter 3} for definitions. The following compactness interpolation theorem is given in \cite{CP1998}*{Theorem 3.1}. 
\begin{prop}\label{CompactnessRealInterpolation}
Let $\mathcal{X}=(\mathcal{X}_0,\mathcal{X}_1)$ and $\mathcal{Y}=(\mathcal{Y}_0,\mathcal{Y}_1)$ be quasi-Banach couples. If $T$ is a bounded linear operator from $\mathcal{X}_1$ to $\mathcal{Y}_1$ and compact from $\mathcal{X}_0$ to $\mathcal{Y}_0$, then $T$ is compact from $(\mathcal{X}_0,\mathcal{X}_1)_{\theta,q}$ to $(\mathcal{Y}_0,\mathcal{Y}_1)_{\theta,q}$ for all $\theta \in (0,1)$ and $q \in (0,\infty]$, 
\end{prop}

\begin{proof}[Proof of Corollary \ref{ContinuousLpCompactness}]
We show the following implications $(b) \Rightarrow (a) \Rightarrow (d) \Rightarrow (b)$ and $(c) \Leftrightarrow (d)$. The first implication $(b) \Rightarrow (a)$ is trivial. 

The implication $(a) \Rightarrow (d)$ follows from a well-known argument. In particular, let 
$$
    k_z^{(p)}(w):=\frac{(1-|z|^2)^{\frac{n+1}{p'}}}{(1-z\overline{w})^{n+1}}.
$$
Notice that $k_z^{(p)}$ are in $\mathcal{A}^p$ with $\|k_z^{(p)}\|_{L^p}\approx 1$, see \cite{IMW2015}*{pp. 1564}. Since the $k_z^{(p)}$ converge weakly to $0$ in $\mathcal{A}^{p}$ as $|z|\rightarrow 1^{-}$ and since $T_u$ is compact on $L^{p}$ (and hence compact on $\mathcal{A}^{p})$, we have $\|T_uk_z^{p}\|_{L^{p}}\rightarrow 0$ as $|z|\rightarrow 1^-$. Hence
$$
    \widetilde{T_u}(z) = |\langle T_uk_z,k_z\rangle|= |\langle T_uk_z^{(p)},k_z^{(p')}\rangle|\leq \|T_uk_z^{(p)}\|_{L^{p}}\|k_z^{(p')}\|_{L^{p'}}\rightarrow 0
$$
as $|z|\rightarrow1^-$.

The implications $(d) \Rightarrow (b)$ and $(d) \Rightarrow (c)$ follow from Lemma \ref{BoundaryValues} and Theorem \ref{ToeplitzLpCompactness}.

It remains to prove $(c) \Rightarrow (d)$. Since $u \in L^{\infty}$, $T_u$ is bounded on, say, $L^4$, and by assumption, $T_u$ is compact from $L^1$ to $L^{1,\infty}$. We therefore have by Proposition \ref{CompactnessRealInterpolation}, that $T_u$ is compact on $L^2$, noting that \cite{BerghLofstrom}*{Theorem 5.3.1} gives
$$
    (L^1,L^4)_{\frac{2}{3},2}=(L^{1,\infty},L^4)_{\frac{2}{3},2}=L^2.
$$ 
The argument we used in showing $(a) \Rightarrow (d)$ in the case $p=2$ completes the proof.
\end{proof}

\subsection{Weighted compactness characterization for reverse-H\"older $B_p$ weights}\label{FullLpCompactnessSubsection}

The work of Rubio de Francia has revolutionized the theory of weighted inequalities, most famously with his extrapolation theorem. Although extrapolation is typically phrased in terms of Muckenhoupt $A_p$ weights, the version of extrapolation in \cite{RdF1984}*{Theorem 3} contains the following statement for $B_p$ weights as a special case. See also \cite{CUMP2011}*{Chapter 2}.
\begin{prop}\label{BpBoundednessExtrapolation}
Let $T$ be a linear operator and $p_0 \in [1,\infty)$. If $T$ is bounded on $L_{\sigma_0}^{p_0}$ for all $\sigma_0 \in B_{p_0}$, then $T$ is bounded on $L_{\sigma}^p$ for all $p \in (1,\infty)$ and all $\sigma \in B_{p}$.
\end{prop}
\noindent Pott and Reguerra proved a version of Proposition \ref{BpBoundednessExtrapolation} for $B_p$ weights on the upper half plane in \cite{PR2013}*{Proposition 4.4}. While their quantitative result was necessary for studying optimal dependence on $B_p$ characteristics, we will only need the version we have stated.

We recall the definition of Calder\'on's complex interpolation spaces. 
Let $\mathcal{X}_1$ and $\mathcal{X}_2$ form a Banach couple. A Banach space $\mathcal{X}$ is an intermediate space between $\mathcal{X}_1$ and $\mathcal{X}_2$ if 
$$
    \mathcal{X}_1\cap\mathcal{X}_2\subseteq \mathcal{X} \subseteq \mathcal{X}_1+\mathcal{X}_2
$$
with continuous inclusions, and is an interpolation space between $\mathcal{X}_1$ and $\mathcal{X}_2$ if additionally any linear mapping on $\mathcal{X}_1+\mathcal{X}_2$ which is bounded on $\mathcal{X}_j$ for $j=0,1$ is also bounded on $\mathcal{X}$. Let $S=\{z \in \mathbb{C}: 0<\text{Re}(z)<1\}$. Define $F(\mathcal{X}_0,\mathcal{X}_1)$ to be the space of all functions from $\overline{S}$ into $\mathcal{X}_0+\mathcal{X}_1$ with the properties that each $f \in F(\mathcal{X}_0,\mathcal{X}_1)$ is bounded and continuous on $\overline{S}$, analytic in $S$ (where $f$ analytic means that $\phi\circ f$ is analytic for any bounded linear functional $\phi$ on $\mathcal{X}_0+\mathcal{X}_1$), and the functions $y \mapsto f(j+iy)$, $j=0,1$, are continuous from $\mathbb{R}$ into $\mathcal{X}_j$. Then $F(\mathcal{X}_0,\mathcal{X}_1)$ is a Banach space with respect to the norm
$$
    \|f\|_{F(\mathcal{X}_1,\mathcal{X}_2)}:=\max\left\{\sup_{y \in \mathbb{R}}\|f(iy)\|_{\mathcal{X}_0},\sup_{y \in \mathbb{R}}\|f(1+iy)\|_{\mathcal{X}_1}\right\}.
$$
Given $\theta \in [0,1]$, the interpolation space $[\mathcal{X}_0,\mathcal{X}_1]_{\theta}$ is the Banach space consisting of all $x \in \mathcal{X}_0+\mathcal{X}_1$ such that $x=f(\theta)$ for some $f \in F(\mathcal{X}_0,\mathcal{X}_1)$ with norm 
$$
    \|x\|_{[\mathcal{X}_0,\mathcal{X}_1]_{\theta}}:=\inf\{\|f\|_{F(\mathcal{X}_0,\mathcal{X}_1)}: x=f(\theta)\}.
$$

The following interpolation theorem was proved by Cwikel and Kalton in \cite{CK1995}. 
\begin{prop}\label{CompactnessInterpolation}
Let $(\mathcal{X}_0,\mathcal{X}_1)$ and $(\mathcal{Y}_0,\mathcal{Y}_1)$ be Banach couples and $T$ be a linear operator that is bounded from $\mathcal{X}_0+\mathcal{X}_1$ to $\mathcal{Y}_0+\mathcal{Y}_1$ and from $\mathcal{X}_j$ to $\mathcal{Y}_j$ for $j=0,1$. If $T$ is compact from $\mathcal{X}_1$ to $\mathcal{Y}_1$, then $T$ is compact from $[\mathcal{X}_0,\mathcal{X}_1]_{\theta}$ to $[\mathcal{Y}_0,\mathcal{Y}_1]_{\theta}$ for any $\theta \in (0,1)$, provided at least one of the following conditions is satisfied:
\begin{enumerate}
\addtolength{\itemsep}{0.2cm}
    \item $\mathcal{X}_1$ has the unconditional martingale differences (UMD) property,
    \item $\mathcal{X}_1$ is reflexive and $\mathcal{X}_1=[\mathcal{X}_0,E]_{\alpha}$ for some Banach space $E$ and some $\alpha \in (0,1)$,
    \item $\mathcal{Y}_1=[\mathcal{Y}_0,F]_{\beta}$ for some Banach space $F$ and some $\beta \in (0,1)$, or
    \item $\mathcal{X}_0$ and $\mathcal{X}_1$ are complexified Banach lattices of measurable functions on a common measure space.
\end{enumerate}
\end{prop}
\noindent Clearly, condition (4) of Proposition \ref{CompactnessInterpolation} is satisfied when $\mathcal{X}_j= \mathcal{A}^{p_j}_{\sigma_j}$ for any $p_j \in [1,\infty)$ and any weight $\sigma_j$, so we are able to interpolate compactness in our setting.

We have the following version of the Stein-Weiss interpolation with change in measure.
\begin{prop}\label{SteinWeissInterpolation}
If $p_0, p_1 \in [1,\infty)$, $\theta \in (0,1)$, and $\sigma_0,\sigma_1$ are weights on $\mathbb{B}_n$ with $\sigma_0 \in B_{p_0}$ and $\sigma_1 \in B_{p_1}$, then
$$
    [\mathcal{A}_{\sigma_0}^{p_0},\mathcal{A}_{\sigma_1}^{p_1}]_{\theta}=\mathcal{A}_{\sigma}^p, 
$$
where 
$$
    \frac{1}{p}=\frac{\theta}{p_0}+\frac{1-\theta}{p_1} \quad\quad\text{and}\quad\quad \sigma^{\frac{1}{p}}=\sigma_0^{\frac{\theta}{p_0}}\sigma_1^{\frac{1-\theta}{p_1}}.
$$
Moreover, the norms of the spaces $[\mathcal{A}_{\sigma_0}^{p_0},\mathcal{A}_{\sigma_1}^{p_1}]_{\theta}$ and $\mathcal{A}_{\sigma}^p$  are equivalent.
\end{prop}
\begin{proof}
We follow the general outline of the argument in \cite{AP2015}*{Theorem 5.1}, with some modifications. Let $\mathcal{X}_\theta=[\mathcal{A}_{\sigma_0}^{p_0},\mathcal{A}_{\sigma_1}^{p_1}]_{\theta}.$ First, with $\sigma_0, \sigma_1,$ and $\sigma$ as defined above, straightforward computations using H\"{o}lder's inequality show that $\sigma \in B_p.$ The containment $\mathcal{X}_{\theta} \subseteq \mathcal{A}_{\sigma}^p$ follows from the fact that we have the isometric inclusions $\mathcal{A}^{p_j}_{\sigma_j} \subseteq L^{p_j}_{\sigma_j}$ for $j=0,1$, properties of complex interpolation, and the fact that
$$
    [L_{\sigma_0}^{p_0},L_{\sigma_1}^{p_1}]_{\theta}=L_{\sigma}^p
$$
with equal norms, which is a version of the Stein-Weiss interpolation result, see \cite{BerghLofstrom}*{Theorem 5.5.3}. In particular, for $f \in \mathcal{X}_{\theta},$ we have the norm inequality $\|f\|_{L^p_\sigma}=\|f\|_{\mathcal{A}^p_\sigma} \leq \|f\|_{\mathcal{X}_\theta}.$

We now want to show that the reverse containment holds continuously; that is, there exists $C>0$ such that $\|f\|_{\mathcal{X}_\theta} \leq C \|f\|_{\mathcal{A}^p_\sigma}$ for any $f \in \mathcal{A}^p_\sigma$. By Stein-Weiss interpolation, we know there exist functions $F_{\zeta}$ so that the mapping $F_{(\cdot)}: \overline{S} \rightarrow L^{p_0}_{\sigma_0}+ L^{p_1}_{\sigma_1} $ has the following properties:
\begin{enumerate}
\item $F_{(\cdot)}$ is bounded and continuous on $\overline{S};$ 
\item $F_{(\cdot)}$ is analytic on $S$;
\item The mappings $y \mapsto F_{j+ \mathrm{i}y}$ are continuous for $j=0,1$, and there holds
$$
    \|F_{\zeta}\|_{L^{p_0}_{\sigma_0}} \leq  \|f\|_{L^p_{\sigma}}
$$ 
for all $\zeta$ with $\text{Re}(\zeta)=0$ and 
$$
    \|F_{\zeta}\|_{L^{p_1}_{\sigma_1}} \leq  \|f\|_{L^p_{\sigma}}
$$ 
for all $\zeta$ with $\text{Re}(\zeta)=1$;
\item $F_\theta=f.$
\end{enumerate}

Define $G_{\zeta}=P(F_{\zeta}),$ for $\zeta \in \overline{S}.$ By Theorem B, we deduce the following properties of $G_{(\cdot)}: \overline{S} \rightarrow \mathcal{A}^{p_0}_{\sigma_0}+ \mathcal{A}^{p_1}_{\sigma_1} $: 
\begin{enumerate}
\item $G_{(\cdot)}$ is bounded and continuous on $\overline{S};$ 
\item $G_{(\cdot)}$ is analytic on $S$;
\item The mappings $y \mapsto G_{j+ \mathrm{i}y}$ are continuous for $j=0,1$, and there exists a positive constant $C$ so that 
$$
    \|G_{\zeta}\|_{\mathcal{A}^{p_0}_{\sigma_0}} \leq C \|f\|_{\mathcal{A}^p_{\sigma}}
$$ 
for all $\zeta$ with $\text{Re}(\zeta)=0$ and 
$$
    \|G_{\zeta}\|_{\mathcal{A}^{p_1}_{\sigma_1}} \leq C \|f\|_{\mathcal{A}^p_{\sigma}}
$$ 
for all $\zeta$ with $\text{Re}(\zeta)=1$;
\item $G_{\theta}=f$.
\end{enumerate}
Indeed, to see that (4) holds, note that by definition $G_{\theta}=Pf$ and since $f \in \mathcal{A}^p_{\sigma}$, we have $f \in L^1$ as $\sigma$ is a $B_p$ weight. Since $P$ reproduces holomorphic functions in $L^1$ (see \cite{Ru1980}*{Theorem 3.1.3}), (4) follows. This shows that $\|f\|_{\mathcal{X}_\theta} \leq C \|f\|_{\mathcal{A}^p_\sigma}$ and completes the proof. 
\end{proof}

The following lemma is the final ingredient needed for the compactness extrapolation. 
\begin{lemma}\label{ExtrapolationLemma}
Let $p, p_0 \in (1,\infty)$ and $r>1$. If $\sigma \in B_p\cap\text{RH}_r$ and $\sigma_0 \in B_{p_0}\cap \text{RH}_r$ satisfy $\sigma^{1-p'},\sigma_0^{1-p_0'} \in \text{RH}_r$, then there exists $p_1 \in (1,\infty)$, $\sigma_1 \in B_{p_1}$, and $\theta \in (0,1)$ such that
$$
    \mathcal{A}_\sigma^p=[\mathcal{A}_{\sigma_0}^{p_0},\mathcal{A}_{\sigma_1}^{p_1}]_{\theta}.
$$
\end{lemma}
\begin{proof}
We closely follow the estimates of \cite{HL2020}*{Lemma 4.4}. Solving the equations from Proposition \ref{SteinWeissInterpolation} for $p_1$ and $\sigma_1$, we must show that there exists $\theta \in (0,1)$ such that 
$$
    p_1=p_1(\theta)=\frac{1-\theta}{\frac{1}{p}-\frac{\theta}{p_0}} \in (1,\infty)
$$
and
$$
    \sigma_1=\sigma_1(\theta)=\sigma^{\frac{p_1}{p(1-\theta)}}\sigma_0^{-\frac{p_1\theta}{p_0(1-\theta)}} \in B_{p_1}.
$$
Since $p_1(\theta)$ is continuous at $0$ and $p_1(0)=p \in (1,\infty)$, we have that $p_1(\theta) \in (1,\infty)$ for sufficiently small $\theta>0$. 

To show $\sigma_1(\theta) \in B_{p_1}$ for some sufficiently small $\theta>0$, fix $K \in \mathcal{D}$. For $\varepsilon,\delta>0$ to be chosen later, estimate
\begin{align*}
    \langle\sigma_1\rangle_{\widehat{K}}\left\langle \sigma_1^{1-p_1'}\right\rangle_{\widehat{K}}^{p_1-1}&=\left\langle\sigma^{\frac{p_1}{p(1-\theta)}}\sigma_0^{-\frac{p_1\theta}{p_0(1-\theta)}}\right\rangle_{\widehat{K}}\left\langle\sigma^{-\frac{p_1'}{p(1-\theta)}}\sigma_0^{\frac{p_1'\theta}{p_0(1-\theta)}}\right\rangle_{\widehat{K}}^{p_1-1}\\
    &= \left\langle \sigma^{\frac{p_1}{p(1-\theta)}}\left(\sigma_0^{-\frac{1}{p_0-1}}\right)^{\frac{p_1\theta}{p_0'(1-\theta)}}\right\rangle_{\widehat{K}}\left\langle\left(\sigma^{-\frac{1}{p-1}}\right)^{\frac{p_1'}{p'(1-\theta)}}\sigma_0^{\frac{p_1'\theta}{p_0(1-\theta)}}\right\rangle_{\widehat{K}}^{p_1-1}\\
    &\leq \left\langle \sigma^{r(\theta)}\right\rangle_{\widehat{K}}^{\frac{1}{1+\varepsilon}}\left\langle\left(\sigma_0^{-\frac{1}{p_0-1}}\right)^{s(\theta)}\right\rangle_{\widehat{K}}^{\frac{\varepsilon}{1+\varepsilon}}\left\langle\left(\sigma^{-\frac{1}{p-1}}\right)^{t(\theta)}\right\rangle_{\widehat{K}}^{\frac{p_1-1}{1+\delta}}\left\langle\sigma_0^{u(\theta)}\right\rangle_{\widehat{K}}^{\frac{\delta(p_1-1)}{1+\delta}},
\end{align*}
where we have used H\"older's inequality with $1+\varepsilon, (1+\varepsilon)'$ and $1+\delta,(1+\delta)'$ and
\begin{align*}
    r(\theta):=\frac{p_1(1+\varepsilon)}{p(1-\theta)}, \,\,\,\,\,
    s(\theta):=\frac{p_1\theta(1+\varepsilon)}{p_0'(1-\theta)\varepsilon}, \,\,\,\,\,
    t(\theta):=\frac{p_1'(1+\delta)}{p'(1-\theta)}, \,\,\,\,\,\text{and}\,\,\,\,\,
    u(\theta):=\frac{p_1'\theta(1+\delta)}{p_0(1-\theta)\delta}.
\end{align*}
Setting $\varepsilon = \frac{\theta p}{p_0'}$ and $\delta=\frac{\theta p'}{p_0}$, we have
$$
    r(\theta)=s(\theta)=\frac{p_1(\theta)(p_0'+\theta p)}{pp_0'(1-\theta)}\quad\,\,\text{and}\quad\,\, t(\theta)=u(\theta)=\frac{p_1(\theta)'(p_0+\theta p')}{p'p_0(1-\theta)}.
$$
Since $r(\theta)$ and $t(\theta)$ are continuous at $0$ and $r(0)=t(0)=1$, we may choose $\theta>0$ small enough such that $\max(r(\theta),t(\theta))\leq r$. We continue estimating the above quantity using our reverse-H\"older hypotheses and the $B_p, B_{p_0}$ conditions of $\sigma,\sigma_0$ by a constant times
\begin{align*}
    \langle\sigma\rangle_{\widehat{K}}^{\frac{r(\theta)}{1+\varepsilon}}
    &\left\langle \sigma_0^{-\frac{1}{p_0-1}}\right\rangle_{\widehat{K}}^{\frac{s(\theta)\varepsilon}{1+\varepsilon}}\left\langle\sigma^{-\frac{1}{p-1}}\right\rangle_{\widehat{K}}^{\frac{t(\theta)(p_1-1)}{1+\delta}}\langle\sigma_0\rangle_{\widehat{K}}^{\frac{u(\theta)\delta(p_1-1)}{1+\delta}}\\
    &=\left(\langle \sigma\rangle_{\widehat{K}}\left\langle\sigma^{-\frac{1}{p-1}}\right\rangle_{\widehat{K}}^{p-1}\right)^{\frac{p_1}{p(1-\theta)}}\left(\langle \sigma_0\rangle_{\widehat{K}}\left\langle \sigma_{0}^{-\frac{1}{p_0-1}}\right\rangle_{\widehat{K}}^{p_0-1}\right)^{\frac{\theta p_1}{p_0(1-\theta)}} \\
    &\leq [\sigma]_{B_p}^{\frac{p_1}{p(1-\theta)}}[\sigma_0]_{B_{p_0}}^{\frac{\theta p_1}{p_0(1 -\theta)}}.
\end{align*}
Taking a supremum over $K \in \mathcal{D}$ completes the proof.
\end{proof}

\begin{proof}[Proof of Theorem \ref{CompactnessExtrapolation}]
Let $p \in (1,\infty)$ and $\sigma \in B_p\cap\text{RH}_r$ with $\sigma^{1-p'}\in\text{RH}_r$. By Lemma \ref{ExtrapolationLemma}, there exist $p_1 \in (1,\infty)$, $\sigma_1 \in B_{p_1}$, and $\theta \in (0,1)$ such that 
$$
    \mathcal{A}_{\sigma}^p=[\mathcal{A}_{\sigma_0}^{p_0},A_{\sigma_1}^{p_1}]_{\theta}.
$$
By Proposition \ref{BpBoundednessExtrapolation}, $T$ acts boundedly on $\mathcal{A}_{\sigma_1}^{p_1}$. Thus, by Proposition \ref{CompactnessInterpolation}, we have that $T$ acts compactly on $\mathcal{A}_{\sigma}^p$ as required.
\end{proof}


\begin{proof}[Proof of Theorem \ref{LpCompactnessCharacterization}]
For the forward direction, since $T_u$ is bounded on $\mathcal{A}_{{\sigma}}^p$ for all ${\sigma} \in B_p$ (by Theorem B) and since $T_u$ is compact on $\mathcal{A}_{\sigma}^p$, by Theorem \ref{CompactnessExtrapolation}, we have that $T_u$ is compact on $\mathcal{A}_{{\sigma}}^q$ for all $q \in (1,\infty)$ and all ${\sigma} \in B_q \cap \text{RH}_r$. In particular, $T_u$ is compact on $\mathcal{A}^2$. Therefore, by Theorem A, we have that $\widetilde{T}_u(z)\rightarrow 0$ as $|z|\rightarrow 1^-$.

For the reverse direction, Theorem A gives that $T_u$ is compact on $\mathcal{A}^2$. Since $T_u$ is bounded on $\mathcal{A}_{{\sigma}}^2$ for all ${\sigma} \in B_2$ (by Theorem B), Theorem \ref{CompactnessExtrapolation} implies that $T_u$ is compact on $\mathcal{A}_{\sigma}^p$.
\end{proof}

\begin{proof}[Proof of Corollary \ref{LpCompactnessCharacterizationCorollary}]
This follows immediately since $\sigma\in B_p$ satisfying \eqref{DyadicWeightRegularity} implies that $\sigma \in \text{RH}_r$ for some $r>1$ by Theorem \ref{ReverseHolder}.
\end{proof}

\section{Weighted bounds beyond B\'ekoll\`e-Bonami weights}\label{BoundednessSection}
\subsection{Weighted $L^p$ bounds}\label{LpBoundsSection}
The following pointwise bound of the Bergman projection was obtained in \cite{RTW2017}:
$$
    |Pf(z)|\lesssim \sum_{K\in \mathcal{D}}\langle |f|\rangle_{\widehat{K}}\chi_{\widehat{K}}(z)
$$
for $f \in L^1$ and almost every $z \in \mathbb{B}_n$. From this bound, one immediately has the following proposition. 
\begin{prop}\label{SparseToeplitz}
If $u \in L^{\infty}$, then
$$
    |T_uf(z)|\lesssim \sum_{K \in \mathcal{D}}\langle |uf|\rangle_{\widehat{K}}\chi_{\widehat{K}}(z),
$$
for all $f \in L^1$ and almost every $z \in \mathbb{B}_n$.
\end{prop}
For $\ell \in \{1,2,\ldots,M\}$, $u \in L^{\infty}(\mathbb{B}_n)$, and $f \in L^1$, set
$$
    S_{u,\ell}f:=\sum_{K \in \mathcal{D}_{\ell}}\langle |uf|\rangle_{\widehat{K}}\chi_{\widehat{K}}
    \quad\text{and}\quad
    S_uf:=\sum_{\ell=1}^{M} S_{u,\ell}f.
$$
Proposition \ref{SparseToeplitz} lets us work with the dyadic operator $S_u$ in place of $T_u$ in our arguments.

\begin{proof}[Proof of Theorem \ref{ToeplitzLpBoundedness}]
Suppose that $p\ge 2$ and set $\sigma'=\sigma^{1-p'}$. We use the equivalence 
$$
    \|T_u\|_{L^p_{\sigma}\rightarrow L^p_{\sigma}} = \|T_u(\cdot\,\sigma')\|_{L^p_{\sigma'}\rightarrow L^p_{\sigma}}
$$
and proceed by duality. Let $f\in L^p_{\sigma'}$ and $g\in L^{p'}_{\sigma}$ be nonnegative functions. Apply 
Proposition \ref{SparseToeplitz} 
to obtain the estimate
$$
    \langle |T_u(f\sigma')|, g\sigma\rangle 
    \lesssim \langle S_{u}(f\sigma'),g\sigma\rangle 
    \leq\sum_{K \in \mathcal{D}}\|u\|_{L^{\infty}(\widehat{K})}\langle f\sigma'\rangle_{\widehat{K}}\langle g\sigma\rangle_{\widehat{K}}|\widehat{K}|.
$$

Using the $u B_p$ condition for $\sigma$, \eqref{SparseCondition}, and the containment $K\subseteq \widehat{K}$, we have
\begin{align*}
    \langle S_{u}(f\sigma'),g\sigma\rangle 
    &\leq\sum_{K \in \mathcal{D}}\|u\|_{L^{\infty}(\widehat{K})}\langle f\sigma'\rangle_{\widehat{K}}\langle g\sigma\rangle_{\widehat{K}}|\widehat{K}|\\
    &=\sum_{K \in \mathcal{D}}\|u\|_{L^{\infty}(\widehat{K})}\frac{\sigma(\widehat{K})\sigma'(\widehat{K})^{p-1}}{|\widehat{K}|^p}\frac{|\widehat{K}|^{p-1}}{\sigma(\widehat{K})\sigma'(\widehat{K})^{p-1}}\int_{\widehat{K}}f\sigma'\,dV \int_{\widehat{K}}g\sigma\,dV\\
    &\leq [\sigma]_{u B_p} \sum_{K\in\mathcal{D}}\left(\frac{1}{\sigma'(\widehat{K})}\int_{\widehat{K}}f\sigma' dV\right)\left(\frac{1}{\sigma(\widehat{K})}\int_{\widehat{K}}g\sigma\,dV\right)|\widehat{K}|^{p-1}\sigma'(\widehat{K})^{2-p}\\
    &\lesssim [\sigma]_{u B_p}\sum_{K\in\mathcal{D}}\left(\frac{1}{\sigma'(\widehat{K})}\int_{\widehat{K}}f\sigma' dV\right)\left(\frac{1}{\sigma(\widehat{K})}\int_{\widehat{K}}g\sigma\,dV\right)|K|^{p-1}\sigma'(K)^{2-p}.
\end{align*}
By H\"older's inequality, we have
$$
    |K|\leq \sigma(K)^{\frac{1}{p}}\sigma'(K)^{\frac{1}{p'}},
$$
and so
$$
    |K|^{p-1}\sigma'(K)^{2-p}\leq
    \sigma(K)^{\frac{p-1}{p}}\sigma'(K)^{\frac{p-1}{p'}}\sigma'(K)^{2-p}
    =\sigma(K)^{\frac{1}{p'}}\sigma'(K)^{\frac{1}{p}},
$$
since $\frac{p-1}{p'}+2-p=\frac{1}{p}$.
Using the above estimates, H\"older's inequality, the disjointness of the kubes $K \in \mathcal{D}_{\ell}$, and Lemma \ref{WeightedMaximal}, we bound $\langle S_{u}(f\sigma'),g\sigma\rangle$ by a constant times
\begin{align*}
    \nonumber
    &[\sigma]_{u B_p}\sum_{\ell=1}^{M}\sum_{K\in\mathcal{D}_{\ell}}\left(\frac{1}{\sigma'(\widehat{K})}\int_{\widehat{K}}f\sigma' dV\right)\left(\frac{1}{\sigma(\widehat{K})}\int_{\widehat{K}}g\sigma\,dV\right)\sigma(K)^{\frac{1}{p'}}\sigma'(K)^{\frac{1}{p}}\\
    \nonumber
     &\quad\leq [\sigma]_{u B_p}\sum_{\ell=1}^{M}\left(\sum_{K\in\mathcal{D}_{\ell}}\left(\frac{1}{\sigma'(\widehat{K})}\int_{\widehat{K}}f\sigma' dV\right)^p\sigma'(K)\right)^{\frac{1}{p}}\left(\sum_{K\in\mathcal{D}_{\ell}}\left(\frac{1}{\sigma(\widehat{K})}\int_{\widehat{K}}g\sigma\,dV\right)^{p'}\sigma(K)\right)^{\frac{1}{p'}}\\
    &\quad\lesssim [\sigma]_{u B_p}\|M_{\sigma'}f\|_{L^p_{\sigma'}}\|M_{\sigma}g\|_{L^{p'}_{\sigma}}\\
    &\quad\lesssim [\sigma]_{u B_p}\|f\|_{L^p_{\sigma'}}\|g\|_{L^{p'}_{\sigma}}.
\end{align*}
The case $1<p<2$ follows from duality since $\sigma \in u^{p-1}B_p$ if and only if $\sigma' \in uB_{p'}$, and $[\sigma']_{uB_{p'}}=[\sigma]_{u^{p-1}B_p}^{\frac{p'}{p}}$. Thus
$$
    \|T_u\|_{L^p_{\sigma}\rightarrow L^p_{\sigma}} = \|T_u^*\|_{L^{p'}_{\sigma'}\rightarrow L^{p'}_{\sigma'}}\lesssim [\sigma']_{u B_{p'}}=[\sigma]_{u^{p-1} B_p}^{\frac{p'}{p}}.
$$
Indeed, note that the adjoint $T_u^*$ (with respect to the $L^2$ pairing) admits the following pointwise bound:
\begin{align*}
|T_u^*f(z)| & = |\overline{u}(z) Pf(z)|\\
& \lesssim \sum_{K \in \mathcal{D}} \langle |f| \rangle_{\widehat{K}} \chi_{\widehat{K}}(z) |u(z)|\\
& \leq \sum_{K \in \mathcal{D}} \|u\|_{L^\infty(\widehat{K})} \langle |f| \rangle_{\widehat{K}} \chi_{\widehat{K}}(z),
\end{align*}
which shows that we can apply the same reasoning to the adjoint $T_u^*.$
\end{proof}

\subsection{Weighted weak-type $(1,1)$ bound}\label{WeaktypeSubsection}

Our methods of proving weighted weak-type $(1,1)$ bounds rely on Proposition \ref{SparseToeplitz} and modifications of arguments from \cite{DsLR2016}.
\begin{thm}\label{SparseGeneralWeakType}
Let $u \in L^{\infty}$ and $\sigma$ be a weight. If $\Phi$ is a Young function with complementary function $\Psi$ such that there exists $C>0$ with
\begin{align}\label{YoungCondition}
    c_{\Phi}:=1+\sum_{k=1}^{\infty}\frac{1}{\Psi^{-1}(\rho^{-C^k})}<\infty,
\end{align}
then 
$$
    \|T_{u}f\|_{L^{1,\infty}_{\sigma}}\lesssim c_{\Phi} \|f\|_{L^1_{M_{u,\Phi}\sigma}}
$$
for all $f \in L^1_{M_{u,\Phi}\sigma}$. 
\end{thm}



\begin{proof}
Let $f \in L^1_{M_{u,\Phi}\sigma}$ be nonnegative. By Proposition \ref{SparseToeplitz}, it suffices to prove the estimate with $S_{u,\ell}$ for $\ell=1,\ldots,M$ in place of $T_u$. Fix $C \in (1,\frac{1}{\rho})$ and let $k_0 \in \mathbb{N}$ be such that $\frac{1}{C}\sum_{k=k_0}^{\infty}(C^{-\frac{k}{2}}+C^{-k})<1$. We can further reduce to proving
$$
    \sigma(\{C<S_{u,\ell}f\leq 2C\}) \lesssim c_{\Phi}\|f\|_{L^1_{M_{u,\Phi}\sigma}},
$$
since then 
$$
    \sigma(\{S_{u,\ell}f>\lambda\})=\sum_{k=0}^{\infty}\sigma\left(\left\{C<S_{u,\ell}\left(\frac{Cf}{2^{k}\lambda}\right)\leq 2C\right\}\right) \lesssim \frac{c_{\Phi}}{\lambda}\|f\|_{L^1_{M_{u,\Phi}\sigma}}
$$
for any $\lambda>0$. Notice that Lemma \ref{FeffermanStein} and \eqref{MaximalPointwise} imply
$$
    \sigma(\{M(uf)>C^{-k_0}\}) \leq \sigma\left(\left\{M_{u}f>C^{-k_0}\right\}\right)\lesssim \|f\|_{L_{M_{u}\sigma}^1}\leq c_{\Phi}\|f\|_{L_{M_{u,\Phi}\sigma}^1},
$$
and so it suffices to estimate $\sigma(\mathcal{E})$, where 
$$
    \mathcal{E}:=\{C<S_{u,\ell}f\leq 2C\}\setminus \{M(uf)>C^{-k_0}\}.
$$

For $k\ge k_0$, define $\mathcal{S}_{u,\ell}^{k}:=\{K \in \mathcal{D}_{\ell}: C^{-k-1} < \langle |u|f\rangle_{\widehat{K}} \leq C^{-k}\}$ and set
$$
    S_{u,\ell}^{k}f(z):= \sum_{K \in \mathcal{S}_{u,\ell}^{k}}\langle |u|f\rangle_{\widehat{K}}\chi_{\widehat{K}}(z).
$$
Note that if $z \in \mathcal{E}$, then 
$$
    S_{u,\ell}f(z)=\sum_{k=k_0}^{\infty} S_{u,\ell}^{k}f(z).
$$
Therefore, it suffices to prove the following claim: there exists $C'>0$ such that 
\begin{equation}\label{GeneralWeakTypeClaim}
    \int_{\mathcal{E}}S_{u,\ell}^{k}f(z)\sigma(z)\,dV(z) \leq (C^{-k}+C^{-\frac{k}{2}})\sigma(\mathcal{E})+ \frac{C'}{\Psi^{-1}(\rho^{-C^k})}\|f\|_{L^1_{M_{u,\Phi}\sigma}}
\end{equation}
for each $k \ge k_0$. Indeed, assuming \eqref{GeneralWeakTypeClaim}, we have 
\begin{align*}
    \sigma(\mathcal{E})&\leq \frac{1}{C}\sum_{k=k_0}^{\infty}\int_{\mathcal{E}}S_{u,\ell}^{k}f(z)\sigma(z)\,dV(z)\\
    &\leq \frac{1}{C}\sum_{k=k_0}^{\infty}\left(C^{-k}+C^{-\frac{k}{2}}\right)\sigma(\mathcal{E})+ C'c_{\Phi}\|f\|_{L_{M_{u,\Phi}\sigma}^1}
\end{align*}
which implies $\sigma(\mathcal{E}) \lesssim c_{\Phi}\|f\|_{L_{M_{u,\Phi}\sigma}^1}$ and completes the proof.

It remains to prove \eqref{GeneralWeakTypeClaim}. Let $\mathcal{S}_{u,\ell}^{k,0}$ be the collection of $K\in\mathcal{S}_{u,\ell}^{k}$ such that $\widehat{K}$ is maximal with respect to inclusion in $\{\widehat{K}:K \in \mathcal{S}_{u,\ell}^k\}$, and inductively let $\mathcal{S}_{u,\ell}^{k,v}$ be the collection of $K \in \mathcal{S}_{u,\ell}^{k}$ such that $\widehat{K}$ is maximal in $\{\widehat{K}: K \in \mathcal{S}_{u,\ell}^{k}\setminus\bigcup_{j=0}^{v-1}\mathcal{S}_{u,\ell}^{k,j}\}$. Then $\mathcal{S}_{u,\ell}^{k}=\bigcup_{v=0}^{\infty}\mathcal{S}_{u,\ell}^{k,v}$ 
and therefore we should estimate
$$
    \int_{\mathcal{E}}S_{u,\ell}^{k}f(z)\sigma(z)\,dV(z) = \sum_{v=0}^{\infty}\sum_{K \in \mathcal{S}_{u,\ell}^{k,v}}\langle |u|f \rangle_{\widehat{K}}\sigma(\mathcal{E}\cap \widehat{K}).
$$

For $K \in \mathcal{S}_{u,\ell}^{k,v}$, set 
$$
    E_K:= \widehat{K}\setminus \bigcup_{K' \in \mathcal{S}_{u,\ell}^{k,v+1}}\widehat{K'}.
$$
By construction, $E_{K}\cap E_{K'}=\emptyset$ for distinct $K,K' \in \mathcal{S}_{u,\ell}^k$. Further, using the fact that $C^{-k-1}< \langle |u|f\rangle_{\widehat{K}} \leq C^{-k}$ for $K \in \mathcal{S}_{u,\ell}^{k}$ and \eqref{SparseCondition}, we have for each $K \in \mathcal{S}_{u,\ell}^{k}$ that
\begin{align*}
    \int_{\widehat{K}} |u(z)|f(z)\,dV(z) &= \int_{E_K}|u(z)|f(z)\,dV(z)+\sum_{K' \in \text{ch}_{\mathcal{S}_{u,\ell}^{k}}(K)}\int_{\widehat{K'}}|u(z)|f(z)\,dV(z)\\
    &\leq \int_{E_K}|u(z)|f(z)\,dV(z) + C^{-k}\sum_{K' \in \text{ch}_{\mathcal{S}_{u,\ell}^{k}}(K)}|\widehat{K'}|\\
    &\leq \int_{E_K}|u(z)|f(z)\,dV(z) + C^{-k}\rho|\widehat{K}|\\
    &\leq \int_{E_K}|u(z)|f(z)\,dV(z) + C\rho\int_{\widehat{K}}|u(z)|f(z)\,dV(z),
\end{align*}
where $\text{ch}_{\mathcal{S}_{u,\ell}^{k}}(K):=\{K' \in \mathcal{S}_{u,\ell}^{k}: \widehat{K'} \subsetneq \widehat{K} \,\, \text{and} \,\, \widehat{K'} \,\,\text{is maximal}\}$. Since $C\rho<1$, this implies that $\int_{\widehat{K}} |u(z)|f(z)\,dV(z) \lesssim \int_{E_K}|u(z)|f(z)\,dV(z)$. 

For fixed $v\ge 0$ and $K \in \mathcal{S}_{u,\ell}^{k,v}$, set 
$$
    \widetilde{K}:= \bigcup_{\substack{K' \in \mathcal{S}_{u,\ell}^{k,v+\lceil C^{k/2}\rceil}\\ \widehat{K'} \subseteq \widehat{K}}}\widehat{K'},
$$
and note that, by repeatedly applying \eqref{SparseCondition}, we have
$$
    |\widetilde{K}|\leq \rho^{C^{\frac{k}{2}}}|\widehat{K}|.
$$
Decompose 
$$
    \widehat{K}=\widetilde{K} \cup \bigcup_{l=0}^{\lfloor C^{k/2}\rfloor}\bigcup_{\substack{K' \in \mathcal{S}_{u,\ell}^{k,v+l}\\ \widehat{K'} \subseteq \widehat{K}}}E_{K'}.
$$

Using the above facts and Lemma \ref{OrliczHolderApplied}, we see that 
\begin{align*}
    \langle |u|f\rangle_{\widehat{K}}\sigma(\mathcal{E}\cap \widetilde{K})&=\int_{\widehat{K}}|u(z)|f(z)\,dV(z) \langle \sigma\chi_{\mathcal{E}\cap \widetilde{K}}\rangle_{\widehat{K}}\\
    &\lesssim \int_{E_K}f(z)\|u\|_{L^{\infty}(\widehat{K})}\frac{\|\sigma\chi_{\mathcal{E}\cap\widetilde{K}}\|_{\Phi,\widehat{K}}}{\Psi^{-1}\left(|\widehat{K}|/|\widetilde{K}|\right)}\,dV(z)\\
    &\lesssim \frac{1}{\Psi^{-1}\left(\rho^{-C^{\frac{k}{2}}}\right)}\int_{E_K}f(z)M_{u,\Phi}\sigma(z)\,dV(z).
\end{align*}
Using the disjointness of the $E_K$ and the above estimate, we sum 
\begin{align*}
    \sum_{v=0}^{\infty}\sum_{K\in\mathcal{S}_{u,\ell}^{k,v}}\langle |u|f\rangle_{\widehat{K}}\sigma(\mathcal{E}\cap \widetilde{K})&\lesssim \frac{1}{\Psi^{-1}\left(\rho^{-C^{\frac{k}{2}}}\right)}\sum_{v=0}^{\infty}\sum_{K\in\mathcal{S}_{u,\ell}^{k,v}}\int_{E_K}f(z)M_{u,\Phi}\sigma(z)\,dV(z)\\
    &\leq \frac{1}{\Psi^{-1}\left(\rho^{-C^{\frac{k}{2}}}\right)}\|f\|_{L_{M_{u,\Phi}\sigma}^1}.
\end{align*}

On the other hand, since $\langle |u|f\rangle_{\widehat{K}}\leq C^{-k}$ for $K \in \mathcal{S}_{u,\ell}^{k}$ and the sets $E_K$ are pairwise disjoint, we have
\begin{align*}
    \sum_{v=0}^{\infty}\sum_{K \in \mathcal{S}_{u,\ell}^{k,v}}\sum_{l=0}^{\lfloor C^{k/2}\rfloor}\sum_{\substack{K' \in \mathcal{S}_{u,\ell}^{k,v+l}\\\widehat{K'}\subseteq \widehat{K}}} \langle |u|f\rangle_{\widehat{K}}\sigma(\mathcal{E}\cap E_{K'})&\leq C^{-k}\lceil C^{k/2}\rceil\sum_{v=0}^{\infty}\sum_{K \in \mathcal{S}_{u,\ell}^{k,v}}\sigma(\mathcal{E}\cap E_{K})\\ 
    &\leq (C^{-\frac{k}{2}}+C^{-k})\sigma(\mathcal{E}).
\end{align*}
This establishes the claim \eqref{GeneralWeakTypeClaim}.
\end{proof}

\begin{cor}\label{SparseMrWeakType}
Let $u\in L^{\infty}$ and $\sigma$ be a weight. If $r \in (1,\infty)$, then 
$$
    \|T_{u}f\|_{L^{1,\infty}_{\sigma}}\lesssim \left(1+\log r'\right)\|f\|_{L^1_{M_{u,r}\sigma}}
$$
for all $f \in L^1_{M_{u,r}\sigma}$, where $\displaystyle M_{u,r}f:=\sup_{K \in \mathcal{D}}\|u\|_{L^{\infty}(\widehat{K})}\langle |f|^r\rangle_{\widehat{K}}^{1/r}\chi_{\widehat{K}}$.
\end{cor}

\begin{proof}
Take $\Phi(t)=t^r$ in Theorem \ref{SparseGeneralWeakType}. We will show that the constant $c_{\Phi}$ in \eqref{YoungCondition} is proportional to $1+\log r'$. We first sketch an argument showing $\Psi^{-1}(t) \approx t^{\frac{1}{r'}}.$ Let $c>2.$ Since $\Psi$ is increasing, to prove $\Psi^{-1}(t) \leq c t^{\frac{1}{r'}}$, it suffices to show $\Psi(c t^{\frac{1}{r'}}) \geq t$. A calculus computation establishes that 
$$
    \Psi\left(ct^{\frac{1}{r'}}\right)= t\left(c\left(\frac{c}{r}\right)^{\frac{1}{r-1}}-\left(\frac{c}{r}\right)^{\frac{r}{r-1}}\right).
$$
Straightforward calculations verify that since $c>2$, there holds, for $r  \in (1, \infty)$:
$$ 
    \left(c\left(\frac{c}{r}\right)^{\frac{1}{r-1}}-\left(\frac{c}{r}\right)^{\frac{r}{r-1}}\right) \geq 1,
$$
as required. Replacing $c$ with $1$, one can show that similarly 
$$
    \Psi\left(t^{\frac{1}{r'}}\right)= t\left(\left(\frac{1}{r}\right)^{\frac{1}{r-1}}-\left(\frac{1}{r}\right)^{\frac{r}{r-1}}\right)
$$
and that for $r \in (1, \infty),$
$$ \left(\left(\frac{1}{r}\right)^{\frac{1}{r-1}}-\left(\frac{1}{r}\right)^{\frac{r}{r-1}}\right) \leq 1,$$
which shows $\Psi(t^{\frac{1}{r'}}) \leq t$ and hence $\Psi^{-1}(t) \geq t^{\frac{1}{r'}}.$

Using the fact that $\Psi^{-1}(t) \approx t^{\frac{1}{r'}}$ and the notation of Theorem \ref{SparseGeneralWeakType}, we estimate:
$$ 
    c_{\Phi}=1+\sum_{k=1}^{\infty}\frac{1}{\Psi^{-1}(\rho^{-C^k})} \approx 1+\sum_{k=1}^{\infty} \rho^{\frac{C^k}{r'}}.
$$
We will use the facts that $\rho<1$ and $C>1$. Note that $\rho^{\frac{C}{r'}}\leq 1,$ so we can ignore the first term in the summation. Interpreting the summation as a Riemann sum and making a couple of simple substitutions, we have
\begin{align*} \sum_{k=2}^{\infty} \rho^{\frac{C^k}{r'}} & \leq \int_{1}^{\infty} \rho^{\frac{C^x}{r'}} \mathop{d x}\\
& = \int_{\frac{C}{r'}}^{\infty} \frac{\rho^u}{(\log{C}) u} \mathop{d u}\\
& = \int_{\frac{C}{r'}}^{\infty} \frac{e^{(\log{\rho})u}}{(\log{C}) u} \mathop{d u}\\
& = \frac{1}{\log{C}} \int_{\frac{-C \log{\rho}}{r'}}^{\infty} \frac{e^{-v}}{v} \mathop{d v}.
\end{align*}
Then, using the estimate 
$$
e^x \int_{x}^{\infty} \frac{e^{-t}}{t} \mathop{dt} \leq \log \left(1+ \frac{1}{x} \right)
$$
from \cite{AS1964} and noting that $r' \in (1, \infty)$, this last integral is dominated by
$$
    \frac{e^{C \log{\rho}/r'}}{\log{C}} \log \left(1- \frac{r'}{C \log{\rho}}\right)  \lesssim \log \left(1- \frac{r'}{C \log{\rho}}\right) \lesssim 1+\log{r'}
$$
which establishes the result.
\end{proof}

\begin{proof}[Proof of Theorem \ref{ToeplitzL1Boundedness}]
Using the fact that $\sigma\in uB_1 \cap \text{RH}_r$, we have
$$
    M_{u,r}\sigma \lesssim [\sigma]_{\text{RH}_r} M_{u}\sigma \lesssim [\sigma]_{uB_1}[\sigma]_{\text{RH}_r} \sigma.
$$ 
Thus, using Corollary \ref{SparseMrWeakType}, we obtain
$$
    \|T_uf\|_{L_{\sigma}^{1,\infty}}\lesssim (1+\log r') \|f\|_{L_{M_{u,r}\sigma}^1}\leq [\sigma]_{uB_1}[\sigma]_{\text{RH}_r}(1+\log r')\|f\|_{L_{\sigma}^1}
$$
as required.
\end{proof}

\begin{proof}[Proof of Corollary \ref{ToeplitzL1BoundednessCorollary}]
Take $r=1+\frac{1}{2\alpha[\sigma]_{B_\infty}}$ and compute $r' = 1+ 2\alpha[\sigma]_{B_{\infty}}$. Applying Theorem \ref{ToeplitzL1Boundedness} and Theorem \ref{ReverseHolder}, we conclude
\begin{align*}
\|T_u\|_{L^{1}_{\sigma}\rightarrow L^{1,\infty}_{\sigma}} &\lesssim [\sigma]_{uB_1}[\sigma]_{\text{RH}_r}(1+\log r')\\
&\lesssim [\sigma]_{uB_1}c_{\sigma}^{ \frac{1}{2\alpha[\sigma]_{B_\infty}+1}}\log(e+[\sigma]_{B_\infty}).
\end{align*}
\end{proof}


\end{document}